\documentclass{alggeom}
\usepackage[utf8]{inputenc}
\usepackage[T1]{fontenc}
\usepackage{lmodern}
\usepackage{graphicx}
\usepackage{longtable}
\usepackage{wrapfig}
\usepackage{rotating}
\usepackage[normalem]{ulem}
\usepackage{amsmath}
\usepackage{amssymb}
\usepackage{capt-of}
\usepackage{hyperref}
\usepackage{etoolbox}
\usepackage{microtype}
\usepackage{mathrsfs}
\usepackage[english]{babel}

\usepackage[backend=biber,style=numeric,isbn=false,doi=false,url=false,eprint=false]{biblatex}
\renewbibmacro{in:}{}

\usepackage{tikz}
\usepackage{tikz-cd}
\usetikzlibrary{decorations.pathmorphing}

\usepackage{amsthm}
\usepackage{thmtools}
\usepackage{thm-restate}
\usepackage{xpatch}
\usepackage{cleveref}
\usepackage{leftindex}

\usepackage{nicematrix}

\declaretheorem[numberwithin=section]{theorem}  
\declaretheorem[sibling=theorem]{corollary}
\declaretheorem[sibling=theorem]{lemma}
\declaretheorem[sibling=theorem]{proposition}
\declaretheorem[sibling=theorem]{definition}
\numberwithin{equation}{section}

\theoremstyle{definition} 
\newtheorem{example}{Example}[section]
\newtheorem{remark}{Remark}[section]
\newtheorem{facts}{Facts}[section]
\newtheorem{notation}{Notation}[section]

\DeclareMathOperator\id{id}
\DeclareMathOperator\bl{Bl}

\DeclareMathOperator\im{im}

\DeclareMathOperator\pic{Pic}

\newcommand\dualab\hat
\newcommand\sh\mathscr
\newcommand\bb\mathbb

\DeclareMathOperator{\gr}{gr}
\renewcommand{\le}{\leqslant}
\renewcommand{\ge}{\geqslant}
\renewcommand{\subset}{\subseteq}

\DeclareMathOperator\Sym{Sym}
\DeclareMathOperator\Sec{Sec}
\DeclareMathOperator\pH{\leftindex^p{\mathcal H}}
\DeclareMathOperator\prim{\mathrm{prim}}
\DeclareMathOperator\Res{\mathrm{Res}}
\DeclareMathOperator\can{can}
\DeclareMathOperator\var{var}
\DeclareMathOperator*\colim{colim}

\makeatletter
\newcounter{proofstep}
\xpretocmd{\proof}{\setcounter{proofstep}{0}}{}{}
\newcommand{\proofstep}[1]{%
  \par
  \addvspace{\medskipamount}%
  \stepcounter{proofstep}%
  \noindent\emph{Step \theproofstep: #1}\par\nobreak\smallskip
  \@afterheading
}
\makeatother

\author{Daniel Brogan}
\address{Mathematics Department\\Stony Brook University\\ Stony Brook, New York}

\title[Invariants of the singularities of secant varieties of curves]{Invariants of the singularities of secant varieties of curves}

\classification{14F06, 14F08, 14F43, 14F45}
\keywords{secant variety, perverse sheaf, intersection cohomology, nearby cycles, vanishing cycles}
\thanks{This paper was partially supported by NSF award DMS-2301526 and also by Simons Foundation International, LTD.}

\hypersetup{
 pdflang={English},
 final,
 bookmarksnumbered=true,
 bookmarksdepth=subsubsection}
\begin{document}

\begin{abstract}
    Consider a smooth projective curve and a given embedding into projective space via a sufficiently positive line bundle. We can form the secant variety of $k$-planes through the curve. These are singular varieties, with each secant variety being singular along the previous one. We study invariants of the singularities for these varieties. In the case of an arbitrary curve, we compute the intersection cohomology in terms of the cohomology of the curve. We then turn our attention to rational normal curves. In this setting, we prove that all of the secant varieties are rational homology manifolds, meaning their singular cohomology satisfies Poincar\'e duality. We then compute the nearby and vanishing cycles for the largest nontrivial secant variety, which is a projective hypersurface.
\end{abstract}
\maketitle

\section{Introduction}\label{sec:Intro}
For a smooth projective curve embedded into projective space one can form the secant variety of $k$-planes. If the embedding is sufficiently positive, each secant variety will be a proper subvariety of projective space which is singular exactly along the next smallest secant variety. In this paper we study invariants of the singularities of these secant varieties. In particular we compute their intersection cohomology and, in the case of a rational normal curve of even degree, in which case the largest nontrivial secant variety is a hypersurface, we compute the nearby and vanishing cycle sheaves. We also study the question of which secant varieties for which curves are rational homology manifolds. Throughout the paper we mostly work in the language of perverse sheaves, however almost all results in this paper automatically lift to the category of pure or mixed Hodge modules.\\
\indent For a variety $X$ of dimension $n,$ the intersection complex $\mathit{IC}_X$ is in general smaller than the shifted constant sheaf $\bb Q_{X}[n].$ In the language of Hodge theory, $\mathit{IC}_X$ is pure of weight $n$ and is the the top graded piece of $\bb Q_{X}[n].$ There is a natural map $\bb Q_{X}[n]\to \mathit{IC}_X$ and $X$ is called a \textbf{rational homology manifold} if this map is an isomorphism. This is true, for example, when $X$ has finite quotient singularities. In general, the difference between $\bb Q_{X}[n]$ and $\mathit{IC}_X$ can be understood as a measure of the singularities of $X.$\\
\indent Now suppose $f:X\to\bb C$ is a holomorphic function. The nearby and vanishing cycle sheaves $\psi_f \bb Q_X[n]$ and $\varphi_f\bb Q_{X}[n]$ are intimately related to the topology of the hypersurface $X_0=f^{-1}(0).$ They give a refined kind of measure of the singularities of $X_0$ and they contain the information of the monodromy of the nearby fibers of the function $f$ around the origin in $\bb C.$ However, explicitly applying these sheaves is notoriously difficult to do except in certain circumstances. To give some examples, this is done in the cases of a normal crossings divisor in \cite{Ste76}, \cite{Sai88}, and more recently in \cite{Che22}, as well as for a generic determinant, see \cite{LY24}. Despite secant varieties not even being local complete intersections and their defining functions being rather complicated determinants, we are still able to get a complete description of $\psi_f \bb Q_X[n]$ and $\varphi_f\bb Q_{X}[n].$\\
\indent Another motivation for studying secant varieties of rational normal curves, and indeed the original motivation for this project, is their relationship to theta divisors on Jacobians of hyperelliptic curves. This relationship can be seen in two ways. Consider a hyperelliptic curve $C$ of genus $g$ and the theta divisor $\Theta$ on the Jacobian $J(C).$ The first relation is that for a point $x\in \Theta$ of multiplicity $m,$ the tangent cone $TC_x\Theta\subset \bb C^g$ is isomorphic to the cone over the topmost secant variety of a rational normal curve of degree $2m$ \cite[Appendix A]{SY22}. Thus the study of secant varieties is in a sense a local study of theta divisors on hyperelliptic Jacobians. The second relation is via resolutions of singularities. In \cite{Ber92} a log resolution is constructed for the pair $(\bb P^N,X),$ where $X$ is a secant variety of a curve. In \cite{SY22} a log resolution of the pair $(J(C),\Theta)$ is constructed in a similar fashion, and it turns out that the fibers of this resolution are exactly the analogous resolutions for secant varieties of rational normal curves.\\
\indent Let $X$ be a complex manifold of dimension $n,$ $K$ a perverse sheaf on $X,$ and $f:X\to\bb C$ a holomorphic function on $X$ which has an isolated critical value at $0.$ Then one can form $\psi_f K$ and $\varphi_f K,$ the nearby and vanishing cycles of $K$ with respect to $f.$ Roughly speaking these are perverse sheaves on the singular fiber $X_0=f^{-1}(0)$ which measure the behavior of $K$ near $X_0$ in a way that is more refined than just taking the restriction $K|_{X_0}.$\\
\indent When $K=\bb Q_X[n],$ the nearby and vanishing cycles give subtle information about the singularities of $X_0.$ This idea is used, in particular, in Saito's definition of mixed Hodge modules. Given a candidate Hodge module $M$ on $X$ one needs to check certain regularity conditions along all holomorphic functions $f$ on $X,$ and this is done via the functors $\psi_f$ and $\varphi_f.$ More recently, the Hodge theoretic nearby and vanishing cycles have found applications in the study of singularities via Hodge ideals \cite{MP20} and higher multiplier ideals \cite{SY23}. Having explicit descriptions of $\psi_f K$ and $\varphi_f K$ more easily allows one to understand exactly what information it contains regarding the singularities of $X.$\\
\indent The paper is in two main sections. The first deals with secant varieties of arbitrary curves, and the second focuses on the case of rational normal curves.

\subsection{Secant varieties of arbitrary curves}\label{subsec:IntroGenCurve}
We begin in \Cref{sec:Prelim} by constructing secant varieties $\Sec^k$ and secant bundles $B^k$ for an arbitrary smooth projective curve $C.$ After developing these preliminaries we move on to studying the intersection complex of each $\Sec^k.$ Sections \ref{subsec:IH} to \ref{subsec:ComputingIH} are devoted to computing the intersection cohomology of the secant varieties $\Sec^k.$ This is the main result of this section.
\begin{theorem}\label{thm:IHSecPrelim}
	Let $C$ be a smooth projective curve, $M$ a line bundle on $C$ which separates $2k$ points, and $\zeta$ the class of the tautological line bundle on the $k$-th secant bundle $B^k\to C^{(k)}.$ Then the intersection cohomology of $\Sec^k$ is given by the formula
	\begin{align*}
		\mathit{IH}^{j}(\Sec^k) = \bigoplus_{\max\{j-k,0\}\le 2i}\wedge^{j-2i}H^1(C)\zeta^i,
	\end{align*}
	where $0\le j\le 2k-1.$ The degrees above the middle are obtained by duality.
\end{theorem}
In particular, the intersection cohomology is entirely determined by the cohomology of the curve $C.$ We end with \Cref{subsec:QSec2} in which we compare more explicitly the constant sheaf and the intersection complex for $\Sec^2.$

\subsection{Secant varieties of rational normal curves}\label{subsec:IntroRNC}
The bulk of the paper is contained in \Cref{sec:RNCs}. This section is dedicated to the study of secant varieties of rational normal curves. Here we switch to the simpler notation $S_k$ for $\Sec^k$ in order to distinguish this setting from the case of an arbitrary curve. The ideals of the $S_k$ are generated by the various minors of generic \textit{Hankel matrices}, i.e. matrices of the form
\begin{equation*}
	H_n=
	\begin{pNiceMatrix}
		x_0 & x_1 & x_2 & \cdots & x_j\\
		x_1 & x_2 & \ddots & \ddots & x_{j+1}\\
		x_2 & \ddots & \ddots & \ddots & x_{j+2}\\
		\vdots & \ddots & \ddots & \ddots & \vdots\\
		x_{n-j} & x_{n-j+1} & x_{n-j+2} & \cdots & x_{2n} 
	\end{pNiceMatrix}.
\end{equation*}
We begin with a key lemma about Hankel matrices in \Cref{subsec:Hank}, which allows us to conclude that $S_k$ is locally isomorphic to a product of an affine space and the cone over a smaller secant variety for a rational normal curve of smaller degree. This ``inductive structure'' on the $S_k$'s will be the most important point in the calculation of the vanishing cycles. In \Cref{subsec:NearbyCyc,subsec:MilFibration} we review the basics of nearby and vanishing cycles and their relationship to affine Milnor fibrations in the case of a homogeneous polynomial on affine space. In \Cref{subsec:QisIC} we prove the following.
\begin{theorem}
		Let $C$ be a rational normal curve. Then each nontrivial secant variety $S_k$ satisfies
		\[
			\bb Q_{S_k}[2k-1]\cong \mathit{IC}_{S_k}.
		\]
\end{theorem}
Thus $S_k$ is a rational homology manifold (compare with \cite[Corollary G]{ORS23}), so its singular cohomology satisfies Poincar\'e duality. It is already known that these varieties are rational and have rational singularities, so in some sense the $S_k$ are ``close'' to projective space. The proof we present here is only for the case of a rational normal curve of \textit{even} degree. A proof for arbitrary degrees will appear in the author's dissertation. Sections \ref{subsec:OrdPart} to \ref{subsec:VanCycPf} develop the necessary tools to prove the main theorem.
\begin{theorem}\label{thm:VanCycPrelim}
	Let $f=\det H_n$ and let $X_n$ the cone over $S_n.$
	\begin{enumerate}
		\item All eigenvalues of the monodromy $T:\psi_f \bb Q_{\bb C^{2n+1}}[2n+1]\to \psi_f \bb Q_{\bb C^{2n+1}}[2n+1]$ are of the form $\lambda = e^{2\pi i p/q}$ where $q\in\{1,\ldots,n+1\}$ and $\gcd(p,q)=1.$
		\item For each eigenvalue $\lambda$ of $T,$ the nearby cycle sheaf $\psi_{f,\lambda}\bb Q_{\bb C^{2n+1}}[2n+1]$ is pure of weight $2n.$
		\item If $\lambda = e^{2\pi i p/q}$ is an eigenvalue of $T$ with $q\neq 1,$ then $$\psi_{f,\lambda}\bb Q_{\bb C^{2n+1}}[2n+1]=\mathit{IC}(L_\lambda)$$ where $L_\lambda$ is a rank $1$ local system on $X_{n-q+1}.$
		\item $\varphi_{f,1}\bb Q_{\bb C^{2n+1}}[2n+1]=0,$ so $\psi_{f,1}\bb Q_{\bb C^{2n+1}}[2n+1] = \bb Q_{X}[2n].$
	\end{enumerate}
\end{theorem}
So the nearby and vanishing cycles decompose into a direct sum of intersection complexes of rank $1$ local systems, each of which is supported on some $X_k.$ This is perhaps the simplest nontrivial result that one could hope for. An application of \Cref{thm:VanCycPrelim} will appear in an upcoming work of Schnell and Yang in which they show that a similar result holds for theta divisors on hyperelliptic Jacobians. We end in \Cref{subsec:Dimca} with a way to explicitly compute eigenvectors of the monodromy operator on the nearby cycles.

\subsection*{Acknowledgments}
I am very grateful to my advisor Christian Schnell for introducing me to this topic and his guidance throughout. I thank Ruijie Yang for helpful comments on an earlier draft of this work. I would also like to express gratitude to Mark de Cataldo for helpful conversations which helped get this project started. Finally, I thank Matthew Huynh, Brad Dirks, Yilong Zhang, and the many PhD students at Stony Brook for the various helpful conversations over the course of this project.

\section{Secant varieties of curves}\label{sec:Prelim}

\subsection{Secant bundles and secant varieties}\label{subsec:SecVars}
In this section we construct secant varieties as in \cite{Ber92}. We only review the main points needed for this paper. We use the convention that $\bb P^k(V)$ denotes the projective space of hyperplanes in the vector space $V.$ Let $C$ be a smooth projective algebraic curve over $\mathbb C.$ The $k$-fold symmetric product $C^{(k)}$ is the quotient of $C^k$ by the natural action of the symmetric group $\Sigma_k.$ $C^{(k)}$ is a smooth projective variety of dimension $k$ and its points are the effective divisors of degree $k$ on $C.$

\begin{definition}
We say that a line bundle $M\in \pic C$ \textbf{separates $k$ points} if
\[
    h^0(C,M(-D))=h^0(M,C)-k
\]
for all $D\in C^{(k)}.$
\end{definition}
\begin{example}
    $M$ separates one point if and only if it is basepoint free and $M$ separates two points if and only if it is very ample.
\end{example}
\begin{example}
    The line bundle $\mathcal O_{\bb P^1}(n)$ separates $n+1$ points for $n\ge 0.$
\end{example}
\indent The \textit{universal divisor} $\mathscr D_k$ of $C\times C^{(k)}$ is defined as the image of the embedding
\begin{align*}
    C\times C^{(k-1)}&\to C\times C^{(k)}\\
    (p,D)&\mapsto (p,p+D). 
\end{align*}
Let $\pi_1,\pi_2$ denote the projections to the first and second factors of $C\times C^{(k)}.$ Then we have the following exact sequence
\begin{align*}
    0\to \pi_1^*M \otimes \mathcal O(-\mathscr D_{k})\to \pi_1^*M\to \pi_1^*M\otimes \mathcal O_{\mathscr D_{k}}\to 0
\end{align*}
and when $M$ separates $k$ points this sequence remains exact when pushed down to $C^{(k)}.$ We then define \textbf{the $k$-th secant bundle} of $C$ (with respect to $M$) to be the projective bundle $B^k(M)=\mathbb P((\pi_2)_*(\pi_1^* M\otimes \mathcal O_{\mathscr D_{k}}))$ over $C^{(k)}.$ We may also denote this as $B^k(C)$ when the line bundle $M$ is clear from context, and when there is no danger of confusion we will omit $M$ and $C$ from the notation entirely and simply write $B^k.$ There is a natural map
\[
    \beta_k: B^k(M)\to \mathbb P((\pi_2)_*(\pi_1^* M)) = \mathbb PH^0(C,M)\times C^{(k)}\to \mathbb PH^0(C,M)
\]
whose image (under certain conditions) is the \textbf{variety of secant $(k-1)$-planes} or the \textbf{$k$-th secant variety} of $C$ and is denoted by $\Sec^k(M).$ Again, we will write $\Sec^k(C)$ or simply $\Sec^k$ depending on the context. The notation is such that a particular fiber of $B^k,$ or a particular $(k-1)$-plane in $\Sec^k,$ is determined by choosing $k$ (not necessarily distinct) points on the curve $C.$ If $B^k_D$ denotes the fiber of the map $B^k\to C^{(k)}$ over $D=p_1+\cdots+ p_{k},$ then $\beta_k(B^k_D)$ is the $(k-1)$-plane secant to $C$ at the points in the support of $D$ with the appropriate multiplicities.\\
\indent For $m < k$ there are also natural maps $\alpha_{m,k}$ induced by the addition map $a=a_{m,k}:C^{(m)}\times C^{(k-m)}\to C^{(k)}.$
\[
    \begin{tikzcd}
    B^m\times C^{(k-m)} \arrow[r, "{\alpha_{m,k}}"] \arrow[d, "\beta_m\times \id_{C^{(k-m)}}"'] & B^k \arrow[d,"\beta_k"] \\
    C^{(m)}\times C^{(k-m)} \arrow[r, "a"']            & C^{(k)}
    \end{tikzcd}
\]
The bundle $B^m\times C^{(k-m)}$ over $C^{(m)}\times C^{(k-m)}$ is called \textbf{the $m$-th relative secant bundle}. One can show that these maps satisfy the following compatibility lemma.
\begin{lemma}
    For $m<\ell<k,$ the following diagrams commute:
    \begin{equation}
        \begin{tikzcd}
            B^m\times C^{(k-m)} \arrow[rd, "\pi_{1}"'] \arrow[r, "{\alpha_{m,k}}"] & B^k \arrow[r, "\beta_k"]   & {\mathbb PH^0(C,M)} \\
            & B^m \arrow[ru, "\beta_m"'] &                    
        \end{tikzcd}
    \end{equation}
    \begin{equation}
        \begin{tikzcd}
        B^m\times C^{(\ell-m)}\times C^{(k-\ell)} \arrow[rd, "{(1,a)}"'] \arrow[r, "{(\alpha_{m,\ell},1)}"] & B^{\ell}\times C^{(k-\ell)} \arrow[r, "{\alpha_{\ell,k}}"] & B^{k} \\
        & B^{m}\times C^{(k-m)} \arrow[ru, "{\alpha_{m,k}}"'] & 
        \end{tikzcd}
    \end{equation}
\end{lemma}
What we need from \cite{Ber92} is an understanding of when $\Sec^k$ is the classical $k$-th secant variety, and what structure the maps $\beta_k$ and $\alpha_{m,k}$ have for varying $k$ and $m.$ We summarize the results in the following proposition.
\begin{proposition}\label{prop:ResSings}
    Let $C$ be a smooth curve in $\bb P^N$ embedded via a line bundle $M$ which separates $2k$ points. For each $m\le k,$ let $Z^k_m = \alpha_{m,k}(B^m\times C^{(k-m)})$ and write $U^k = B^k\setminus Z^k_{k-1}.$
    \begin{enumerate}
        \item $B^1$ is isomorphic to the curve $C,$ the secant bundle map $\beta_1:B^1\to C$ is an isomorphism, and with this identification, $\beta_1:B^1\to \bb P^N$ is the embedding into $\bb P^N$ induced by $M.$ In particular, $\Sec^1(C)=C.$ 
        \item For each $m=2,\ldots, k,$ $\Sec^m$ is a proper subvariety of $\bb P^N$ singular along $\Sec^{m-1}.$ Furthermore, the map $\beta_m:B^m\to\Sec^m$ is an isomorphism on $U^m.$ In particular, it is a resolution of singularities with exceptional divisor $\beta_m^{-1}(\Sec^m)=Z^m_{m-1}.$
        \item For each $m\in\{2,\ldots,k\},$ the singular locus of $Z^k_m$ is $Z^k_{m-1}.$ Furthermore, the map $\alpha_{m,k}:B^m\times C^{(k-m)}\to Z^k_m$ is an isomorphism on $U^m\times C^{(k-m)}.$ This map is similarly a resolution of singularities with exceptional divisor $Z^m_{m-1}\times C^{(k-m)}.$
    \end{enumerate}
\end{proposition}
From this point on we will always make the assumption that $M$ separates $2k$ points so that the above proposition applies, unless otherwise stated.

\subsection{Intersection cohomology}\label{subsec:IH}
The majority of this section will be devoted to finding a general formula for the intersection cohomology of $\Sec^k(C)$ for any smooth curve $C$ embedded by a sufficiently positive line bundle. \Cref{subsec:SemiSmall} covers some homological preliminaries about perverse sheaves and semismall maps. In \Cref{subsec:IHSec,subsec:PiAlpha} we study the relevant maps on the cohomology of the secant bundles $B^k.$ The final computation takes place in \Cref{subsec:ComputingIH}. For the rest of the section, we work with cohomology with $\bb Q$-coefficients unless otherwise stated, and we omit the coefficient field from our notation in this case.
\begin{notation}
    Many of the Hodge structures floating around in this section are Tate twisted, sometimes many times. Usually these twists are induced by explicit differential forms, and so to keep track of this while avoiding notation that is too unwieldy, we will write the forms explicitly. For example, if $B^k$ is the $k$-th secant bundle for a curve $C,$ we have the projection map
    \[
        H^j(B^k\times C)\cong \bigoplus_{i=0}^2 H^{j-i}(B^k)\otimes H^i(C)\to H^{j-2}(B^k)\otimes H^2(C)\cong H^{j-2}(B^k)(-1).
    \]
    We will instead write the right hand side as $H^{j-2}(B^k)\omega$ where $\omega\in H^2(C)$ is a generator. This will have the advantage of making the effect of certain maps completely clear.
\end{notation}

\subsection{Semismall maps and the decomposition theorem}\label{subsec:SemiSmall}
Let $f:X\to Y$ be a proper morphism of irreducible complex varieties and define
\[
    Y_m = \{y\in Y\mid \dim f^{-1}(y)=m\}.
\]
We say that $f$ is \textbf{semismall} if
\begin{equation}\label{eqn:SemiSmall}
    2m + \dim Y_m\le \dim X
\end{equation}
for each $m.$ The $Y_m$ for which equality holds in (\ref{eqn:SemiSmall}) are called the \textbf{relevant strata} for $f.$ In \cite{dCM02}, de Cataldo and Migliorini prove an especially useful form of the BBDG decomposition theorem when $X$ is smooth and the morphism in question is semismall. We state a simplified version which will suffice for our purposes.
\begin{theorem}\label{thm:dCMDecomp}
    Let $f:X\to Y$ be a proper semismall morphism between irreducible complex varieties.  Furthermore, assume that $X$ is smooth and the fibers of $f$ are irreducible. Then in the bounded derived category $D_c^b(Y)$ there is a canonical isomorphism
    \[
        Rf_*\bb Q_X[n] \cong \bigoplus_m \mathit{IC}_{\overline{Y_m}},
    \]
    where the sum runs over all relevant strata for $f.$
\end{theorem}

\subsection{Finding the intersection complex}\label{subsec:IHSec}
Now let $C$ be a curve embedded in projective space by a line bundle which separates $2k$ points. We stratify the secant variety $\Sec^k$ by open subsets $U^m$ of the smaller secant varieties:
\[
	U^m=\Sec^m\setminus \Sec^{m-1}\subset \Sec^k
\]
for $m\le k.$ By \Cref{prop:ResSings} we have that the fiber over $x\in U^m$ is
\[
    \beta_k^{-1}(x)\cong C^{(k-m)}.
\]
It follows that for $x\in U^m$
\[
    2\dim \beta_k^{-1}(x) + \dim U^m = 2(k-m) + 2m-1 = 2k-1 = \dim B^k.
\]
Thus $\beta_k$ is a semismall morphism for each $k$ and each stratum is a relevant stratum for $\beta_k.$ Furthermore, the fibers of the maps $\beta_k$ are just symmetric powers of $C$ and hence are irreducible. Thus we can apply \Cref{thm:dCMDecomp} to get a canonical decomposition in the bounded derived category $D^b_c(\Sec^k):$
\begin{align*}
    R(\beta_k)_*\mathbb Q_{B^k}[2k-1]\simeq \bigoplus_{m=1}^k \mathit{IC}_{\Sec^m}.
\end{align*}
\begin{theorem}\label{thm:proj}
    Let $C$ be a curve embedded in projective space by a line bundle separating $2k+2$ points. The map of perverse sheaves on $\Sec^{k+1}$
    \begin{align*}
        \pi_*\alpha^*:R(\beta_{k+1})_*\mathbb Q_{B^{k+1}}[2k+1]\to R(\beta_{k})_*\mathbb Q_{B^{k}}[2k-1]
    \end{align*}
    has $\ker\pi_*\alpha^*=\mathit{IC}_{\Sec^{k+1}},$ where $\pi:B^{k}\times C\to B^{k}$ is the projection and $\alpha=\alpha_{k,k+1}$ is the map on relative secant bundles.
\end{theorem}

Thus if we compute the kernel of $\pi_*\alpha^*$ on the level of cohomology, then we can compute the intersection cohomology of the secant varieties. The main part of the proof of \Cref{thm:proj} is contained in the following proposition.
\begin{proposition}\label{prop:surj}
    The map $\pi_*\alpha^*:H^j(B^{k+1}) \to H^{j-2}(B^{k})\omega$ is surjective for each $j\ge 0.$
\end{proposition}
\begin{proof}[Proof of \Cref{thm:proj}]
    The $\mathit{IC}_{\Sec^m}$ for $1\le m\le k+1$ have distinct supports in $\Sec^{k+1},$ so by irreducibility and strict supports, the map $\pi_*\alpha^*$ decomposes into a sum of maps $\mathit{IC}_{\Sec^m}\to \mathit{IC}_{\Sec^m}$ which are either isomorphisms or zero. In particular, $\mathit{IC}_{\Sec^{k+1}}$ is in the kernel. To see that no other $\mathit{IC}_{\Sec^m}$ is in the kernel, note that the map induced on cohomology
    \[
        H^{2k-2m+2}(B^{k+1},\mathbb Q)\to H^{2k-2m}(B^{k},\mathbb Q)
    \]
    is surjective by \Cref{prop:surj}. Since $\mathbb Q = \mathit{IH}^0(\Sec^m)\subset H^{2k-2m}(B^{k},\mathbb Q)$ the map $\mathit{IC}_{\Sec^m}\to \mathit{IC}_{\Sec^m}$ cannot be zero. Thus $\ker(\pi_*\alpha^*) = \mathit{IC}_{\Sec^{k+1}}.$
\end{proof}
It now suffices to prove \Cref{prop:surj}. To do this we will thoroughly study the maps $\pi_*\alpha^*$ on cohomology.

\subsection{The maps $\pi_*\alpha^*$}\label{subsec:PiAlpha}
On cohomology the map $\pi_*\alpha^*$ is the composite
\begin{equation}\label{eqn:comp}
	\begin{tikzcd}
		H^j(B^{k+1})\ar[r,"\alpha^*"]& H^j(B^{k}\times C)\ar[r,"\pi_*"]& H^{j-2}(B^{k})\omega,
	\end{tikzcd}
\end{equation}
where $\omega\in H^2(C)$ is a generator, the map $\alpha=\alpha_{k,k+1}$ is induced by the addition map $a:C^{(k)}\times C\to C^{(k+1)},$ and the map $\pi_*$ is induced by the projection coming from the K\"unneth formula. Since each $B^k$ is a $\bb P^{k-1}$-bundle over $C^{(k)},$ its cohomology ring $H^*(B^k)$ is generated as an algebra over $H^*(C^{(k)})$ by the class $\zeta$ of the tautological line bundle. In any given degree this just means
\begin{equation}\label{eqn:projbun}
    H^j(B^k)\cong \bigoplus_{i=0}^{k-1} H^{j-2i}(C^{(k)})\zeta^i,
\end{equation}
where by convention we take cohomology in negative degrees to be $0.$ The map the map $\alpha^*$ is induced via the above algebra structure by the addition map
\[
	a: C^{(k)}\times C\to C^{(k+1)}.
\]
The idea is that it should suffice to understand $\pi_*\alpha^*\zeta$ and the effect of $\pi_*a^*$ on the level of $C^{(k+1)}.$\\
\indent We will start with understanding $\alpha^*\zeta.$ We will need the following lemma:
\begin{lemma}\label{lem:TautClass}
    Let $\zeta_{k+1}$ and $\zeta_{k}$ be the tautological classes for $B^{k+1}$ and $B^{k}$ respectively. Then $\alpha^*(\zeta_{k+1})=\pi^*(\zeta_{k}).$
\end{lemma}
The proof of \Cref{lem:TautClass} uses two elementary lemmas. Recall that for a vector bundle $p:E\to S,$ the tautological class $\zeta$ on the projective bundle, which we will denote $\widetilde p:\mathbb{P}(E)\to S,$ comes from the tautological line bundle $\mathcal O_{\mathbb P(E)}(1)$ which is defined by the exact sequence
    \[
        \begin{tikzcd}
            0 \arrow[r] & T_E \arrow[r] & p^*E \arrow[r] & \mathcal O_{\mathbb P(E)}(1) \arrow[r] & 0
        \end{tikzcd}
    \]
    where $T_E$ is the vector bundle whose fiber over $x\in\mathbb P(E)$ is the corresponding hyperplane in $E_{p(x)}.$
\begin{lemma}\label{lem:pullback1}
    Let $p:E\to S$ and $p':E'\to S$ be two vector bundles over a common base $S.$ Suppose that $E'$ is a quotient of $E,$ i.e. we have a commutative diagram
    \[
    \begin{tikzcd}
        E \arrow[rd, "p"'] \arrow[rr, "q", two heads] &   & E' \arrow[ld, "p'"] \\
        & S &                    
    \end{tikzcd}
    \]
    Abusing notation, let $\widetilde q:\mathbb P(E')\to\mathbb P(E)$ denote the map on the projective bundles. Then $\widetilde{q}^*\mathcal O_{\mathbb P(E)}(1)\cong\mathcal O_{\mathbb P(E')}(1).$
\end{lemma}

\begin{lemma}\label{lem:pullback2}
    Let
    \[
    \begin{tikzcd}
        f^*E \arrow[d, "p'"'] \arrow[r, "f'"] & E \arrow[d, "p"] \\
        S' \arrow[r, "f"'] & S              
    \end{tikzcd}
    \]
    be a map of vector bundles over bases $S$ and $S'$ induced by the map $f:S'\to S.$ Furthermore, let $\widetilde f:\mathbb P(f^*E)\to\mathbb P(E)$ denote the induced map on projective bundles. Then $\widetilde f^*\mathcal O_{\mathbb P(E)}(1)\cong \mathcal O_{\mathbb P(f^*E)}(1).$
\end{lemma}
\begin{proof}[Proof of \Cref{lem:TautClass}]
    For each $m<k+1$ we have the following commutative diagram.
    \[
    \begin{tikzcd}
        B^m\times C^{(k+1-m)} \arrow[d, equal] \arrow[rr, "{\alpha_{m,k+1}}"]    & &    B^{k+1}\arrow[d, equal] \\
        \mathbb{P}(\pi_{C^{(m)}}^*E^m) \arrow[r] \arrow[rd]    & \mathbb{P}(a^*E^{k+1}) \arrow[r, "\widetilde a"] \arrow[d] & \mathbb{P}(E^{k+1}) \arrow[d]\\
        & C^{(m)}\times C^{(k+1-m)} \arrow[r, "a"'] & C^{(k+1)}
        \end{tikzcd}
    \]
    Note that the map $a^*E^{k+1}\to\pi_{C^{(m)}}^*E^m$ is a surjection, so we can apply \Cref{lem:pullback1,lem:pullback2} and get that $\alpha_{m,k+1}^*\mathcal O_{B^{k+1}}(1) \cong \mathcal O_{\mathbb{P}(\pi_{C^{(m)}}^*E^m)}(1) = \pi_{C^{(m)}}^*\mathcal O_{B^m}(1).$ The conclusion follows in the case $m=k.$
\end{proof}

Now we turn our attention to the map
\begin{equation}\label{eqn:PiAMap}
    \pi_*a^*: H^j(C^{(k+1)})\to H^{j-2}(C^{(k+1)})\omega.
\end{equation}
\begin{proposition}\label{prop:PiASurj}
    The map in (\ref{eqn:PiAMap}) is surjective for each $j,$ and the kernel is given by
    \begin{equation}\label{eqn:symprodker}
        \ker(\pi_*a^*) \cong
        \begin{cases}
            \wedge^m H^1(C) & 0\le j\le k+1,\\
            0 & \text{otherwise.}
        \end{cases}
    \end{equation}
\end{proposition}
\begin{proof}
    \indent In \cite{Mac62} Macdonald computes the cohomology of the symmetric product $C^{(k+1)}$ in terms of the cohomology of $C.$ In fact, he gives explicit generators. If $p_i:C^k\to C$ for $i=1,\ldots, k$ denote projection onto the various factors, then define
    \begin{align*}
        \overline{\xi}_i &= p_1^*\gamma_i + \cdots + p_{k+1}^*\gamma_i\qquad \text{for } i=1,\ldots,2g,\\
        \overline\eta &= p_1^*\omega + \cdots + p_{k+1}^*\omega,
    \end{align*}
    where $g$ is the genus of $C,$ the $\gamma_i$ generate $H^1(C),$ and $\omega$ generates $H^2(C).$ The cohomology classes $\overline{\xi_i}$ and $\overline\eta$ are invariant under the action of the symmetric group. They therefore descend to cohomology classes on $C^{(k+1)}$ which we denote by $\xi_i$ and $\eta$ respectively. Macdonald shows that these classes generate the cohomology of $C^{(k+1)}.$ He also gives relations between the $\xi_i$ and $\eta$ (see also \cite{Gug17}). In degrees $j\le k+1$ the $\xi_i$ anticommute and $\eta$ is central. Hence for $j\le k+1$ we arrive at the isomorphism
    \begin{equation}\label{eqn:symprodcoho}
        H^j(C^{(k+1)})\cong \bigoplus_{i\ge 0}\left(\wedge^{j-2i}H^1(C)\right)\eta^i.
    \end{equation}
    \indent If $\xi_i'$ and $\eta'$ denote the classes in $H^*(C^{(k)})$ analogous to $\xi_i$ and $\eta$ respectively, then we obviously have (up to perhaps a multiplicative constant) that
    \begin{align}
        a^*\xi_i &= \xi_i'\otimes p_{k+1}^*\gamma_i,\\
        a^*\eta &= \eta'\otimes p_{k+1}^*\omega
    \end{align}
    It follows that, under the the isomorphism in (\ref{eqn:symprodcoho}), the map $\pi_*a^*,$ which is induced by the projection in the K\"unneth formula and the addition map, is just the projection map formally sending $\eta^i$ to $(\eta')^{i-1}\omega.$ Explicitly, in degrees $j=0,\ldots,k+1$ we have a diagram.
    \[
    \begin{tikzcd}
        H^j(C^{(k+1)}) \arrow[rr, "\pi_*a^*"] \arrow[d, equal]     &  & H^{j-2}(C^{(k)})\omega \arrow[d, equal]     \\
        \bigoplus_{i\ge 0}\left(\wedge^{j-2i}H^1(C)\right)\eta^i \arrow[rr] &  & \bigoplus_{i\ge 1}\left(\wedge^{j-2i}H^1(C)\right)(\eta')^{i-1}\omega
    \end{tikzcd}
    \]
    The bottom map is just the projection away from the $i=0$ factor. Therefore it is surjective and its kernel is $\wedge^j H^1(C).$ We can similarly find the kernel in higher degrees using the hard Lefschetz isomorphisms, which we denote by $L^i.$ When $j=k+2$ we have the diagram
    \[
    \begin{tikzcd}
        H^{k+2}(C^{(k+1)}) \arrow[rr, "\pi_*a^*"] \arrow[d, equal, "L"']     &  & H^{k}(C^{(k)})\omega \arrow[dd, equal]     \\
        H^k(C^{(k+1)})\eta \arrow[d, equal]\\
        \bigoplus_{i\ge 0}\left(\wedge^{k-2i}H^1(C)\right)\eta^{i+1} \arrow[rr] &  & \bigoplus_{i\ge 0}\left(\wedge^{k-2i}H^1(C)\right)(\eta')^i\omega
    \end{tikzcd}
    \]
    and the bottom arrow is an isomorphism. Finally in the case $j>k+2$ we have the diagram
    \[
    \begin{tikzcd}
        H^j(C^{(k+1)}) \arrow[rr, "\pi_*a^*"] \arrow[d, equal, "L^{j-k-1}"']     &  & H^{j-2}(C^{(k)})\omega \arrow[d, equal, "L^{j-k-2}"]     \\
        H^{2k+2-j}(C^{(k+1)})\eta^{\ell} \arrow[d, equal] &  & H^{2k+2-j}(C^{(k)})(\eta')^{\ell-1}\omega \arrow[d, equal]\\
        \bigoplus_{i\ge 0}\left(\wedge^{2k+2-j-2i}H^1(C)\right)\eta^{\ell+i} \arrow[rr] &  & \bigoplus_{i\ge 0}\left(\wedge^{2k+2-j-2i}H^1(C)\right)(\eta')^{\ell-1+i}\omega
    \end{tikzcd}
    \]
    and once again the bottom arrow is an isomorphism. To summarize, we have calculated that the map $\pi_*a^*:H^j(C^{(k+1)})\to H^{j-2}(C^{(k)})\omega$ is always surjective and the kernel is given by the isomorphism in (\ref{eqn:symprodker})
\end{proof}

\indent We now have enough information to prove the surjectivity in \Cref{prop:surj}, which will complete the proof of \Cref{thm:proj}.
\begin{proof}[Proof of \Cref{prop:surj}]
    Observe that we have the isomorphism of Hodge structures
    \[
        H^{j-2}(B^{k})\omega\cong \bigoplus_{i=0}^{k-2}H^{j-2-2i}(C^{(k)})\zeta_{k}^i\omega.
    \]
    Take any $\beta\otimes\zeta_{k}^i\otimes\omega\in H^{j-2-2i}(C^{(k)})\zeta_{k}^i\omega$ and let $\gamma\in H^{j-2i}(C^{(k+1)})$ be in $(\pi_*a^*)^{-1}(\beta\otimes \omega).$ Recalling that $\pi_*$ is just the K\"unneth projection
    \[
        H^j(B^{k}\times C)\to H^{j-2}(B^{k})\omega,
    \]
    it then follows that
    \begin{align*}
        \pi_*\alpha^*(\gamma\otimes \zeta_{k+1}^i) = \pi_*(a^*\gamma\otimes \pi^*\zeta_{k}^i) = \beta\otimes\zeta_{k}^i\otimes\omega.
    \end{align*}
\end{proof}

Because each $\bb Q_{B^k}[2k+1]$ and each $\mathit{IC}_{\Sec^k}$ underlie Hodge modules and all of the maps above underlie morphisms of Hodge modules, we automatically get the following corollary.
\begin{corollary}
    Let $C$ be a smooth curve embedded in projective space by a line bundle which separates $2k$ points. Then we have an isomorphism of Hodge modules
    \begin{align*}
        R(\beta_k)_*\mathbb Q_{B^k}[2k-1]\simeq \bigoplus_{m=1}^k \mathit{IC}_{\Sec^m}(-(k-m)).
    \end{align*}
\end{corollary}

\subsection{Computing the intersection cohomology}\label{subsec:ComputingIH}
We can now compute the intersection cohomology of $\Sec^k.$
\begin{theorem}\label{thm:IHSec}
    Let $C$ be a smooth projective curve, $M$ a line bundle on $C$ which separates $2k$ points, and $\zeta$ the class of the tautological line bundle on the $k$-th secant bundle $B^k\to C^{(k)}.$ Then the intersection cohomology of $\Sec^k$ is given by the formula
    \begin{align*}
        \mathit{IH}^{j}(\Sec^k) = \bigoplus_{\max\{j-k,0\}\le 2i}\wedge^{j-2i}H^1(C)\zeta^i,
    \end{align*}
    where $0\le j\le 2k-1.$
\end{theorem}
The degrees above the middle are obtained by duality. In particular, for $j\le k$ we have
\begin{equation}\label{eqn:secvarcoho}
	\mathit{IH}^{j}(\Sec^k) = \bigoplus_{i\ge 0}\left(\wedge^{j-2i}H^1(C)\right)\zeta^i \cong H^j(C^{(k)}).
\end{equation}
\begin{proof}[Proof of \Cref{thm:IHSec}]
    By \Cref{thm:proj} we get a long exact sequence in cohomology.
    \[
    \begin{tikzcd}
        \cdots \arrow[r] & \mathit{IH}^{j}(\Sec^k) \arrow[r] & H^j(B^k) \arrow[r, "\pi_*\alpha^*"] & H^{j-2}(B^{k-1})\omega \arrow[r] & \cdots
    \end{tikzcd}
    \]
    By \Cref{prop:surj} the connecting maps are zero, so $\mathit{IH}^{j}(\Sec^k)$ is the kernel of the map $\pi_*\alpha^*:H^j(B^k)\to H^{j-2}(B^{k-1})\omega.$ Decomposing this map according to the direct sum decompositions in (\ref{eqn:projbun}), this takes the form of a map
    \[
        \pi_*\alpha^*:\bigoplus_{i=0}^{k-1} H^{j-2i}(C^{(k)})\zeta_k^i\to \bigoplus_{i=0}^{k-2}H^{j-2-2i}(C^{(k-1)})\zeta_{k-1}^i\omega.
    \]
    Again we emphasize the distinction between $\zeta_k$ and $\zeta_{k-1}.$ Because $\alpha$ is induced by the addition map $a:C\times C^{(k-1)}\to C^{(k)},$ it can be seen that the components of this map are of the form
     \[
        \pi_*a^*:H^{j-2i}(C^{(k+1)})\zeta_k^i\to H^{j-2-2i}(C^{(k)})\zeta_{k-1}^i\omega.
    \]
    where $\pi$ here also denotes the projection $C^{(k)}\times C\to C^{(k)}.$ Then by \Cref{prop:PiASurj} the kernel of this map is
    \begin{equation*}
        \ker(\pi_*a^*) =
        \begin{cases}
            \wedge^{j-2i} H^1(C)\zeta_k^i & 0\le j-2i\le k,\\
            0 & \text{otherwise.}
        \end{cases}
    \end{equation*}
    It follows that $\mathit{IH}^{j}(\Sec^k)$ is the sum of the above groups for $i=0,\ldots, k-1.$ This is exactly the desired result.
\end{proof}

The formula is worth specifying for the case $C\cong \bb P^1.$
\begin{corollary}\label{cor:P1IH}
    If $C\cong \bb P^1,$ then the intersection cohomology of $\Sec^k$ is
    \[
        \mathit{IH}^{j}(\Sec^k) =
        \begin{cases}
            \bb C & j \text{ even and } 0\le j\le 4k-2,\\
            0 & \text{otherwise.}
        \end{cases}
    \]
\end{corollary}

\subsection{The constant sheaf of $\Sec^2$}\label{subsec:QSec2}
Now we give a strategy for more precisely computing the intersection complex $\mathit{IC}_{\Sec^k},$ carrying out this computation in the case $k=2.$ We make use of a theorem belonging to the study of Du Bois complexes, originally studied in \cite{DBoi81}. An introduction can be found in \cite[Section 7.3]{PS08}. Specifically, we need the following result (see \cite[Example 7.25]{PS08}).
\begin{theorem}\label{thm:DuBois}
	Let $X$ and $Y$ be complex algebraic varieties with $X$ singular along the subvariety $Z.$ Let $p:Y\to X$ be a map which is an isomorphism away from $E=p^{-1}(Z).$
	\[
	\begin{tikzcd}
		E & Y \\
		Z & X
		\arrow[hook, "j", from=1-1, to=1-2,]
		\arrow[from=1-1, to=2-1]
		\arrow["p", from=1-2, to=2-2]
		\arrow["i"', hook, from=2-1, to=2-2]
	\end{tikzcd}
	\]
	Then we have a distinguished triangle
	\[
	\begin{tikzcd}
		\bb Q_{X}\ar[r,"{(p^*,-i^*)}"]& p_*\bb Q_Y\oplus i_*\bb Q_Z\ar[r,"j^*+p^*"]& p_*\bb Q_E\ar[r,"+1"]&\cdots
	\end{tikzcd}
	\]
	in the bounded derived category $D^b_c(X).$
\end{theorem}

\begin{theorem}\label{thm:QSec2}
	Let $C$ be a smooth projective curve embedded by a line bundle which separates $4$ points. Then $\bb Q_{\Sec^2}[3]$ is perverse and there is an exact sequence of perverse sheaves 
	\begin{equation}\label{exseq:QSec2}
		\begin{tikzcd}
			0\ar[r]& \bb Q_C[1]\otimes H^1(C)\ar[r]& \bb Q_{\Sec^2}[3]\ar[r]& \mathit{IC}_{\Sec^2}\ar[r]& 0
		\end{tikzcd}
	\end{equation}
\end{theorem}
\begin{proof}
	By \Cref{prop:ResSings} the diagram
	\begin{equation}
		\begin{tikzcd}
			{B^1\times C} & {B^2}\\
			{C} & {\Sec^2}
			\arrow[hook, from=2-1, to=2-2]
			\arrow["\pi_1"', from=1-1, to=2-1]
			\arrow["\beta_2", from=1-2, to=2-2]
			\arrow[hook, "\alpha^*", from=1-1, to=1-2]
		\end{tikzcd}
	\end{equation}
	satisfies the hypotheses of \Cref{thm:DuBois}, where $\pi_1$ denotes the projection onto the first factor $B^1\cong C.$ Hence we get an exact triangle in the derived category.
	\begin{equation}
		\begin{tikzcd}
			\bb Q_{\Sec^2}[3]\ar[r]& (\beta_2)_*\bb Q_{B^2}[3]\oplus\bb Q_C[1]\ar[r]& (\beta_2)_*\bb Q_{C\times C}[3]\ar[r, "+1"]&\cdots
		\end{tikzcd}
	\end{equation}
	After applying \Cref{thm:dCMDecomp}, the long exact sequence in perverse cohomology sheaves reduces to the exact sequences
	\begin{equation*}
		0\to \pH^{-2}\mathbb Q_{\Sec^2}[3]\to \mathbb Q_{C}[1] \to \bb Q_{C}[1]\otimes H^{0}(C,\bb Q)\to \pH^{-1}\mathbb Q_{\Sec^2}[3]\to 0,
	\end{equation*}
	\begin{equation*}
		0\to \bb Q_{C}[1]\otimes H^{1}(C,\bb Q) \to \pH^{0}\mathbb Q_{\Sec^2}[3]\to \mathit{IC}_{\Sec^2}\oplus \bb Q_{C}[1] \to \bb Q_{C}[1]\otimes H^{2}(C,\bb Q)\to 0.
	\end{equation*}
	Clearly the middle map in the top sequence is an isomorphism, so $\bb Q_{\Sec^2}[3]$ is perverse. In the second sequence, $\mathit{IC}_{\Sec^2}$ has strict support, hence its image is zero. It must therefore be that $\bb Q_C[1]$ maps isomorphically onto $\bb Q_C[1]\otimes H^2(C,\bb Q).$ Thus this sequence contains the exact sequence of perverse sheaves in (\ref{exseq:QSec2}) as a direct summand.
\end{proof}

\begin{corollary}
	The singular cohomology of $\Sec^2$ is given by
	\begin{align*}
		H^0(\Sec^2) &\cong H^0(C^{(2)}),\\
		H^1(\Sec^2) &= 0,\\
		H^2(\Sec^2) &\cong H^0(C^{(2)})\zeta,\\
		H^3(\Sec^2) &= \Sym^2(H^1(C)),\\
		H^4(\Sec^2) &= H^2(C^{(2)})\zeta,\\
		H^5(\Sec^2) &= H^3(C^{(2)})\zeta,\\
		H^6(\Sec^2) &= H^4(C^{(2)})\zeta,
	\end{align*}
	where $\zeta$ is the tautological class for the secant bundle $B^2\to C^{(2)}.$ In particular, $H^3(\Sec^2)$ is pure of weight 2. The other $H^i$ are pure of weight $i.$
\end{corollary}
\begin{proof}
	This follows \Cref{thm:QSec2} after taking the long exact sequence in cohomology. Alternatively, one can use long exact sequence coming from the triangle in (\ref{exseq:QSec2}).
\end{proof}

In particular $\Sec^2(C)$ is never a rational homology manifold unless $C\cong \bb P^1.$ We will see later that in fact all secant nontrivial varieties of rational normal curves are rational homology manifolds.

\section{Secant varieties of rational normal curves}\label{sec:RNCs}

\subsection{Hankel matrices}\label{subsec:Hank}
We now restrict our attention to the case $C\cong \bb P^1$ is a rational normal curve of degree $2n$ in $\bb P^{2n}.$ We will use the more compact notation $S_k=S_k(2n)=\Sec^k(\mathcal O_{C}(2n))$ to denote the secant varieties of $C$ and we write $X_k=X_k(2n)$ to denote the cone of $S_k(2n)$ in $\bb C^{2n+1}.$\\
\indent  It is well known (see \cite[Proposition 4.3]{Eis88}) that the ideal of $S_k$ is generated by the $(k+1)\times (k+1)$ minors of any matrix of the form
\begin{equation*}
	\begin{pNiceMatrix}
		x_0 & x_1 & x_2 & \cdots & x_m\\
		x_1 & x_2 & \ddots & \ddots & x_{m+1}\\
		x_2 & \ddots & \ddots & \ddots & x_{m+2}\\
		\vdots & \ddots & \ddots & \ddots & \vdots\\
		x_{n-m} & x_{n-m+1} & x_{n-m+2} & \cdots & x_{2n} 
	\end{pNiceMatrix}
\end{equation*}
where $n-k\le m\le n.$ For example, the curve $C=S_1$ is the zero locus of the ideal generated by all of the $2\times 2$ minors of the above matrix, $S_2$ is cut out by the $3\times 3$ minors, and so on. Matrices of this form are known as \textbf{Hankel matrices} or \textbf{catalecticant matrices}. To be precise, a Hankel matrix $H$ is a matrix such that $H_{i,j}=H_{i',j'}$ if $i+j=i'+j'.$ We are primarily interested in square Hankel matrices, i.e. matrices as above where $m=n.$ We will denote the $(n+1)\times (n+1)$ Hankel matrix by
\begin{equation*}
	H_n =
	\begin{pNiceMatrix}
		x_0 & x_1 & \cdots & x_n\\
		x_1 & x_2 & \ddots & x_{n+1}\\
		\vdots & \ddots & \ddots & \vdots\\
		x_n & x_{n+1} & \cdots & x_{2n} 
	\end{pNiceMatrix}.
\end{equation*}
The hypersurface $S_n$ of $\bb P^{2n}$ is the largest nontrivial secant variety of $C$ and its defining equation is $f=\det H_n.$ The following fact about Hankel matrices is elementary, but it will be extremely useful for understanding the local geometry of the $S_k.$ It also, to my knowledge, does not appear anywhere in the literature.
\begin{lemma}\label{lem:CoordChange}
	Let $H_{n}$ be as above and let $f=\det H_n$ considered as a function on $\bb C^{2n+1}.$ Fix $k\in \{0,\ldots,n-1\}$ and let
	\[
		Y_k=\{x\in \bb C^{2n+1}\mid x_j=0 \text{ for }j\le k-1\text{ and }x_k\neq 0\}.
	\]
	Then there are coordinates $y_0,\ldots, y_{2n-k}$ on $Y_k$ such that 
	\[
		f|_{Y_k}(y) = y_0^{k+1}\det H_{n-k-1}(y_{k+2},\ldots,y_{2n-k}).
	\]
\end{lemma}
The proof below shows that we can transform the matrix $H_n$ into a block matrix of the form
\begin{equation}\label{eqn:NewHankel}
	\begin{pNiceArray}{*{7}{c}}
		\Block[borders={bottom,right}]{3-3}{} 0 & \cdots & y_0 & \Block{3-4}<\LARGE>{\mathbf 0}\\
		\vdots & \iddots & \vdots &&&&\\
		y_0 & \cdots & y_k &&&&\\
		\Block{4-3}<\LARGE>{\mathbf 0}& & & \Block[borders={top,left}]{4-4}{} y_{k+2} & y_{k+3} & \cdots & y_{n+1}\\
		& & & y_{k+3} & y_{k+4} & \ddots & y_{n+2}\\
		& & & \vdots & \ddots & \ddots & \vdots\\
		& & & y_{n+1} & y_{n+2} & \cdots & y_{2n-k} 
	\end{pNiceArray}
\end{equation}
while keeping the determinant unchanged. In the matrix above, both nonzero blocks are Hankel matrices. In the top left block $A$ we have $A_{i,j}=0$ for $i+j<k.$
\begin{proof}
	For the proof we will let $H=H_n.$ Inductively define functions $p_0,\ldots, p_{2n-k}$ on $Y$ by the identities $p_0x_k=1$ and
	\[
	p_0x_{k+\ell} + p_1x_{k+\ell-1} + \cdots + p_{\ell}x_k = 0
	\]
	for $\ell=1,\ldots,2n-k.$ Now consider the $(n+1)\times (n+1)$ upper triangular matrix
	\[
	P =
	\begin{pmatrix}
		p_0 & p_1 & \cdots & p_n\\
		0 & p_0 & \cdots & p_{n-1}\\
		\vdots & \vdots & \ddots & \vdots\\
		0 & 0 & \cdots & p_0
	\end{pmatrix}.
	\]
	If we start indexing our matrices from $0$ then we have the formulas $H_{ij}=x_{i+j}$ and $P_{ij} = p_{j-i},$ where we take the convention $p_i=0$ for $i<0.$ Consider the product $N=P^THP.$ Then we have
	\begin{align}\label{eqn:NDef}
		N_{ij} &= \sum_{a,b=0}^n(P^T)_{ia}H_{ab}P_{bj}\nonumber\\
		&= \sum_{a,b=0}^n p_{i-a}x_{a+b}p_{j-b}\nonumber\\
		&= \sum_{a=0}^i\sum_{b=0}^j p_{i-a}x_{a+b}p_{j-b}.
	\end{align}
	We aim to show that $N$ is a block diagonal matrix of the form in (\ref{eqn:NewHankel}). We break this up into three cases.\\
	\textbf{Case 1: The top left block.} This is a $(k+1)\times (k+1)$ matrix, so in this case we have $i,j\in\{0,\ldots, k\}.$ We want to show that this is a Hankel matrix whose terms above the main antidiagonal are zero. More precisely, we want to show that
	\[
	N_{i,j} =
	\begin{cases}
		0 & \text{if } i+j = 0,\ldots, k-1,\\
		p_{i+j-k} & \text{if } i+j = k,\ldots,2k.
	\end{cases}
	\]
	If $i+j\in\{0,\ldots,k-1\}$ then each $x_{a+b}$ in (\ref{eqn:NDef}) is zero by assumption, so $N_{i,j}=0.$ Now suppose $i+j=k,\ldots,2k.$ Then
	\begin{align*}
		N_{i,j} &= \sum_{a=0}^i p_{i-a}\sum_{b=0}^{j} x_{a+b}p_{j-b}\\
		&= \sum_{a=0}^i p_{i-a}\left(\sum_{b=0}^{k-1-a} x_{a+b}p_{j-b} + \sum_{b=k-a}^{j} x_{a+b}p_{j-b}\right).
	\end{align*}
	The left $b$-indexed sum contains only $x_{a+b}$ with $a+b\le k-1,$ which all vanish by assumption. The right $b$-indexed sum is exactly the expression defining $p_{j-k+a},$ which vanishes except when $j-k+a=0.$ In that case we have $a=k-j,$ so this term is the only nonzero term of the sum. Therefore
	\[
		N_{i,j}=p_{i+j-k}x_kp_0 = p_{i+j-k},
	\]
	as desired.\\
	\textbf{Case 2: The bottom left block.} In this case we have $i=0,\ldots, k$ and $j=k+1,\ldots,n.$ As in Case 1 we can write
	\begin{align*}
		N_{i,j} &= \sum_{a=0}^i p_{i-a}\sum_{b=0}^{j} x_{a+b}p_{j-b}\\
		&= \sum_{a=0}^i p_{i-a}\left(\sum_{b=0}^{k-1-a} x_{a+b}p_{j-b} + \sum_{b=k-a}^{j} x_{a+b}p_{j-b}\right)
	\end{align*}
	and the only possibly nonzero term occurs when $j-k+a=0.$ However, now we have $j>k$ so that that $j-k+a>0,$ Therefore the sum on the right is zero as well and all entries of the the bottom left block vanish. By symmetry the top right block vanishes as well.\\
	\textbf{Case 3: The bottom right block}. This is a $(n-k)\times (n-k)$ matrix with entries $N_{ij}$ where $i,j=k+1,\ldots, n.$ First we will show that this is a Hankel matrix, then we will compute the entries. To show that this is a Hankel matrix, it is enough to show that $N_{i+1,j}=N_{i,j+1}$ whenever $i,j=k+1,\ldots,n-1.$ Separating the $a=0$ terms from the expression in (\ref{eqn:NDef}), we find that
	\begin{align*}
		N_{i+1,j} &= p_{i+1}\sum_{b=0}^j x_bp_{j-b}+\sum_{a=1}^{i+1}\sum_{b=0}^j p_{i+1-a}x_{a+b}p_{j-b}\\
		&= p_{i+1}\sum_{b=k}^j x_bp_{j-b}+\sum_{a=1}^{i+1}\sum_{b=0}^j p_{i+1-a}x_{a+b}p_{j-b}
	\end{align*}
	Where we've removed the first $k$ terms in the first sum using the fact that $x_b=0$ for $b<k.$ But now the first sum is the expression defining $p_{j-k},$ which is zero. So we have
	\[
	N_{i+1,j} = \sum_{a=1}^{i+1}\sum_{b=0}^j p_{i+1-a}x_{a+b}p_{j-b}.
	\]
	Similarly, by separating the $b=0$ terms from $N_{i,j+1}$ we will find that
	\[
	N_{i,j+1} = \sum_{a=0}^{i}\sum_{b=1}^{j+1} p_{i-a}x_{a+b}p_{j+1-b}.
	\]
	These two expressions are the same by reindexing, so $N_{i+1,j}=N_{i,j+1}.$ Therefore this block is a Hankel matrix.\\
	\indent To compute the entries, it suffices to check the first and last rows. The entries in the first row of this block are of the form $N_{k+1,j}$ for $j=k+1,\ldots,n.$ We have
	\begin{align*}
		N_{k+1,j} &= \sum_{a=0}^{k+1}\sum_{b=0}^j p_{k+1-a}x_{a+b}p_{j-b}\\
		&= \sum_{a=0}^kp_{k+1-a}\sum_{b=0}^jx_{a+b}p_{j-b} + p_0\sum_{b=0}^j x_{k+1+b}p_{j-b}\\
		&= \sum_{a=0}^kp_{k+1-a}\left(\sum_{b=0}^{k-a-1}x_{a+b}p_{j-b} + \sum_{b=k-a}^jx_{a+b}p_{j-b}\right) + p_0\sum_{b=0}^j x_{k+1+b}p_{j-b}\\
	\end{align*}
	The first and second $b$-indexed sums are zero by the same argument as in Case 2 Thus
	\[
	N_{k+1,j} = p_0\sum_{b=0}^j x_{k+1+b}p_{j-b} = -p_0x_kp_{j+1}=-p_{j+1}.
	\]
	by the definitions of $p_0$ and $p_{j+1}.$ The computation for the last row is similar and yields
	\[
	N_{n,j} = -p_{j+n-k}.
	\]
	\indent To conclude, we've shown that 
	\begin{equation*}
		N = \begin{pNiceArray}{*{7}{c}}
			\Block[borders={bottom,right}]{3-3}{} 0 & \cdots & p_0 & \Block{3-4}<\LARGE>{\mathbf 0}\\
			\vdots & \iddots & \vdots &&&&\\
			p_0 & \cdots & p_k &&&&\\ 
			\Block{4-3}<\LARGE>{\mathbf 0}& & & \Block[borders={top,left}]  {4-4}{} -p_{k+2} & -p_{k+3} & \cdots & -p_{n+1}\\
			& & & -p_{k+3} & -p_{k+4} & \ddots & -p_{n+2}\\
			& & & \vdots & \ddots & \ddots & \vdots\\
			& & & -p_{n+1} & -p_{n+2} & \cdots & -p_{2n-k} 
		\end{pNiceArray}.
	\end{equation*}
	Let $N'=p_0^{-2}N.$ By the definition of $N$ we have
	\[
	\det N' = p_0^{-2n-2}(\det P)^2\det H_n = \det H_n = f.
	\]
	On the other hand, by the explicit description for $N$ above, we have
	\[
	\det N' = p_0^{-k-1}\det H_{n-k-1}\left(-p_0^{-2}p_{k+2},\ldots,-p_0^{-2}p_{2n-k}\right).
	\]
	Thus, in the coordinates
	\[
	y_i =
	\begin{cases}
		p_0^{-1}     & i=0,                \\
		p_0^{-2}p_i  & i=1,\ldots, k       \\
		-p_0^{-2}p_i & i=k+1,\ldots, 2n-k,
	\end{cases}
	\]
	we have
	$f=y_0^{k+1}\det H_{n-k-1}(y_{k+2},\ldots, y_{2n-k})$ as desired.
\end{proof}
\indent \Cref{lem:CoordChange} has the following consequence for the local geometry of $S_n.$
\begin{corollary}\label{cor:Nbhd}
	If $C$ is a rational normal curve in $\bb P^{2n},$ then any point $x\in S_n(2n)$ has a Zariski open neighborhood $U\subset S_n(2n)$ such that $U\cong \bb C\times X_{n-1}(2n-2)).$
\end{corollary}
\begin{proof}
	When $x_0\neq 0,$ \Cref{lem:CoordChange} states that the function $f=\det H_n$ takes the form $y_0\det H_{n-1}$ in some coordinates $y_0,\ldots y_{2n}.$ We can dehomogenize by setting $y_0=1,$ so $f=\det H_{n-1}$ on the affine open with coordinates $y_1,\ldots, y_{2n}.$ Since $y_1$ does not appear in $H_{n-1},$ the zero locus of $f$ in this affine open is $\bb C\times X_{n-1}(2n-2).$ Thus the theorem is true for any $x\in S_n(2n)$ with $x_0\neq 0.$\\
	\indent Now we want to show that this works for any $x\in S_n(2n).$ Observe that the coordinates $(x_0,\ldots, x_{2n})$ on $\bb P^{2n} = \bb PH^0(C,\mathcal O_C(2n))^\vee$ are induced by the coordinates $(z,w)$ on $C\cong \bb P^1,$ namely $x_k$ is the coefficient of the form $z^{2n-k}w^k.$ The hyperplane $H=\{x_0=0\}$ is the \textit{osculating hyperplane} of $C$ at the point $p=(0,1),$ i.e. $H$ is the hyperplane such that $H\cap C=2np.$ Suppose $x\in S_n(2n)$ is in the complement of a hyperplane $H'$ such that $H'\cap C = 2np'$ for some $p'\in C.$ Let $T\in SL_2(\bb C)$ be such that $T(H')=H.$ Then we can apply the result for $p$ and transform by $T^{-1}$ to get the result for $p'.$\\
	\indent Now it just suffices to show that any point $x\in \bb P^{2n}$ is in the complement of some osculating hyperplane $H$ of $C.$ This amounts to showing that the sections corresponding to the osculating hyperplanes span $H^0(C,\mathcal O(2n)).$ These sections are the $2n$-th powers of linear forms. It is an elementary fact that any polynomial in one variable of degree $d$ can be written as a sum of $d$-th powers of linear forms. Homogenizing this fact allows us to conclude.
\end{proof}

\subsection{Review of nearby and vanishing cycles}\label{subsec:NearbyCyc}
Now we take some time to review the basics of the nearby and vanishing cycle functors. These functors act on $D^b_c(X),$ the derived category of constructible sheaves on $X,$ in a way which generalizes the vanishing cycles in Picard-Lefschetz theory. A comprehensive introduction on the topology of vanishing cycles and their connection to perverse sheaves can be found in \cite[Chapter 10]{Max19}. For a quick introduction in the case of perverse sheaves, see \cite[Sections 5.5-5.6]{dCM09}.\\
\indent Let $X$ be a complex manifold, let $f:X\to \bb C$ be a holomorphic function on $X$ with an isolated critical value at $0.$ By Ehresmann's theorem, $f$ is a locally trivial fibration away from the origin. Let $X_0=f^{-1}(0)$ be the singular fiber of $f.$ The nearby cycle functor $\psi_f:D^b_c(X)\to D^b_c(X_0)$ is defined as follows. Let $i:X_0\to X$ be the inclusion and let $j:X^*=X\setminus X_0\to X$ be the inclusion of the complement. The exponential map $\exp:\bb C\to \bb C^*$ is the universal cover. Let $\widetilde X$ be the total space of the pullback of the fibration $f|_{X^*}:X^*\to \bb C^*$ via the map $\exp.$ We have a diagram
\[
\begin{tikzcd}
	X_0\ar[r,"i"]\ar[d]& X\ar[d,"f"]& X^*\ar[l,"j"']\ar[d]& \widetilde X\ar[l,"p"']\ar[d]\\
	0\ar[r]& \bb C& \bb C^*\ar[l] & \bb C\ar[l,"\exp"]
\end{tikzcd}
\]
For $K\in D^b_c(X),$ the \textbf{nearby cycles} $\psi_fK\in D^b_c(X_0)$ are defined as
\[
\psi_fK = i^*(j\circ p)_*(j\circ p)^*K[-1]
\]
Evidently $\psi_fK$ depends only on the restriction of $K$ to $X^*.$ By adjunction there is a natural map $K[-1]\to (j\circ p)_*(j\circ p)^*K[-1],$ so applying $i^*$ to this we get a map $i^*K[-1]\to \psi_fK.$ The \textbf{vanishing cycles} $\varphi_fK$ are the cone over this morphism, so that there is a distinguished triangle.
\[
\begin{tikzcd}
	i^*K[-1]\ar[r]& \psi_fK\ar[r,"\can"]& \varphi_fK\ar[r,"+1"]& \cdots
\end{tikzcd}
\]
Note that this is not a definition of $\varphi_f$ as a functor since the cone over morphism in the derived category is not a functorial construction, however this description will suffice for our purposes. The full construction of $\varphi_f$ can be found in \cite[Section 8.6]{KS90}. It is also possible to construct a morphism
\[
\begin{tikzcd}
	\varphi_fK\ar[r,"\var"]& \psi_fK.
\end{tikzcd}
\]
\begin{theorem}[{\cite[Corrolaire 1.6-1.7]{Bry86}}]
	If $K$ is a perverse sheaf on $X,$ then $\psi_fK$ and $\varphi_fK$ are perverse sheaves on $X_0.$
\end{theorem}
\indent Now let $K$ be a perverse sheaf, so that the above theorem applies. The group of deck transformations of the covering $\exp:\bb C\to \bb C^*$ is generated by the map $z\to z+1.$ This induces a map on $\widetilde X$ above, and hence also induces a map $T:\psi_fK\to \psi_fK$ called the \textbf{monodromy}. Since the category of perverse sheaves is an abelian category, we can take the generalized eigenspaces
\[
\psi_{f,\lambda}K = \ker(T-\lambda\id)^N,
\]
where $\lambda\in \bb C^*$ and $N$ is sufficiently large. We then have a direct sum decomposition
\[
\psi_f K = \bigoplus_{\lambda\in\bb C^*}\psi_{f,\lambda} K.
\]
Similarly, we get a monodromy operator on $\varphi_fK$ which we also denote by $T,$ along with a decomposition into generalized eigenspaces. The generalized eigenspaces $\psi_{f,1}K$ and $\varphi_{f,1}K$ are called the \textbf{unipotent parts} of the nearby and vanishing cycles respectively. Since $T$ fixes $i^*K,$ the distinguished triangle above along with the generalized eigenspace decompositions yield an exact sequence of perverse sheaves.
\[
\begin{tikzcd}
	0\ar[r]& i^*K[-1]\ar[r]& \psi_{f,1}K\ar[r,"\can"]& \varphi_{f,1}K\ar[r]& 0
\end{tikzcd}
\]
Moreover, for all $\lambda\neq 1$ the morphism $\can$ induces an isomorphism $\psi_{f,\lambda} K\cong \varphi_{f,\lambda}K.$\\
\indent On the unipotent part $\psi_{f,1}K,$ the nilpotent operator $N=(2\pi i)^{-1}\log T$ is equal to the composition $\var\circ\can.$ Similarly, on $\varphi_{f,1}K$ we have $N=\can\circ\var.$ If we think of $N$ as an operator on $\psi_{f,1}K,$ then $\varphi_{f,1}K\cong \im N$ in the category of perverse sheaves. The nilpotent operator $N$ induces a filtration $W_\bullet$ on $\psi_{f,1}K$ in the following way, see \cite[Lemma 6.4]{Sch73}.
\begin{proposition}\label{prop:WeightFilt}
	Let $N$ be a nilpotent endomorphism on a finite-dimensional complex vector space $V.$ Then there is a unique filtration $W_\bullet$ on $V$ such that
	\begin{enumerate}
		\item for each $k\in \bb Z$ we have $N(W_kV)\subseteq W_{k-2}V,$
		\item for each $k\ge 1$ the map
		\[
		N^k:\gr_k^W V\to \gr_{-k}^W V
		\]
		is an isomorphism.
	\end{enumerate}
\end{proposition}

The filtration $W_\bullet$ is called the \textbf{monodromy weight filtration}, or simply the \textbf{weight filtration}.
\begin{example}\label{ex:N2=0}
	If $N\neq 0$ but $N^2=0$ then the weight filtration on $V$ is $W_1V = V,$ $W_0V =\ker N,$ $W_{-1}V=\im N,$ $W_{-2} V = 0.$
\end{example}
\begin{corollary}\label{cor:N2=0}
	Let $N$ is a nilpotent operator on $V.$ Then $N^2=0$ if and only if the filtration induced on $\im N$ as a quotient of $V,$ or equivalently via the induced action of $N,$ is trivial.
\end{corollary}
\Cref{prop:WeightFilt} hold in any abelian category, so we can apply it to complexes of constructible sheaves. Hence the nilpotent operator $N$ induces a weight filtration on both $\psi_{f,1}K$ and $\varphi_{f,1}K.$ The perverse sheaf $i^*K[-1]$ also gets endowed with a weight filtration by virtue of being a subobject of $\psi_{f,1}K.$ When $K$ underlies a mixed Hodge module, the weight filtration from $N$ and the weight filtration from the MHM structure on $i^*K[-1]$ coincide.

\subsection{The affine Milnor fibration}\label{subsec:MilFibration}
We keep the notation of the previous section, however we now let $X=\bb C^n$ and $f:X\to \bb C$ a \textit{homogeneous polynomial} on $X$ of degree $d.$ For any $k\in \bb N$ let $\mu_k$ denote the group of $k$-th roots of unity. In our setting, the only possibly singular fiber is $X_0=f^{-1}(0).$ The fibration $f_{X_*}:X^*\to \bb C^*$ is called the \textbf{(affine) Milnor fibration} associated to $f$ and we call $F=f^{-1}(1)$ the \textbf{(affine) Milnor fiber}. It can be shown that $F$ is homotopy equivalent to the usual locally defined Milnor fiber at $0,$ for example see \cite[Section 3.1]{Dim92}. Acting by a generator of $\pi_1(\bb C^*),$ we get the \textbf{monodromy transformation} which we also denote by $T:F\to F.$ If $\lambda\in\mu_d$ is a $d$-th root of unity, then $f(x_1,\ldots,x_n)=1$ yields
\[
f(\lambda x_1,\ldots,\lambda x_n) = \lambda^d f(x_1,\ldots,x_n) = 1,
\]
so $\mu_d$ acts on $F$ as well. In fact these actions are the same. Indeed, if $\gamma(t)=\exp(2\pi it)$ is a path which generates $\pi_1(\bb C^*),$ then for a point $x\in F,$ the path $\widetilde\gamma(t)=\gamma(t/d)x$ lifts $\gamma,$ and we have $\widetilde\gamma(0)=x$ and $\widetilde\gamma(1)=\exp(2\pi i/d)x.$\\
\indent We have the following relationship between $F$ and the the nearby cycles $\psi_f.$
\begin{proposition}\label{prop:VanCycMilFib}
	Let $X=\bb C^n$ and let $f:X\to \bb C^n$ be a homogeneous polynomial. Then the cohomology of the stalk of $\psi_f \bb Q_X[n]$ at $0$ is given by the  singular cohomology of $F$ with rational coefficients. Furthermore, the isomorphism commutes with the monodromy.
	\[
	\begin{tikzcd}
		H^k(\psi_f \bb Q_X[n])_0\ar[r,"\cong"]\ar[d,"T"'] & H^{k+n-1}(F,\bb Q)\ar[d,"T"]\\
		H^k(\psi_f \bb Q_X[n])_0\ar[r,"\cong"'] & H^{k+n-1}(F,K|_F)
	\end{tikzcd}
	\]
\end{proposition}
\begin{proof}
	$H^k(\psi_f \bb Q_X[n])_0$ is obtained by taking an appropriate complex representing $\psi_f \bb Q_X[n],$ restricting to $0,$ and taking cohomology of this complex of vector spaces. But by the definition of $\psi_f,$ restricting to $0$ is the same as taking a representative $I$ of $K[-1]$ and computing
	\[
	\colim_{0\in U} \Gamma(I,((j\circ p)^{-1}(X^*\cap U)))
	\]
	where $U$ ranges over all neighborhoods of $0.$ But for small neighborhoods $U$ of $0,$ the open set
	\[
		(j\circ p)^{-1}(X^*\cap U)
	\]
	is homeomorphic to $\widetilde X,$ which deformation retracts to $F.$ Hence this colimit is just $\Gamma(F,I|_{F}).$ The cohomology of this is then
	\[
		H^{k}(F,\bb Q_F[n-1])=H^{k+n-1}(F,\bb Q_F).
	\]
	The statement about the monodromy follows since in both cases it is induced by the deck transformations of $\exp:\bb C\to\bb C^*.$
\end{proof}

\subsection{Purity of the constant sheaf}\label{subsec:QisIC}
In order to compute the nearby and vanishing cycles, we will need to compute the cohomology of the Milnor fiber. To do this, it will be useful to understand the relationship between $\bb Q_{S_k}[2k-1]$ and $\mathit{IC}_{S_k}$ in a way similar to \Cref{thm:QSec2} above. In fact, for rational normal curves they are the same. We present the proof for even degrees here, since it will suffice for our purposes and the presence of \Cref{lem:CoordChange} allows us to simplify the proof greatly. Nonetheless, the statement holds for rational normal curves of any degree. The proof of the general case will appear in the author's dissertation.
\begin{theorem}\label{thm:PureQSec}
	Let $C=\mathbb P^1\subset \bb P^{2n}$ be a rational normal curve of degree $2n.$ For each $k=1,\ldots, n,$ let $X_k$ be the affine cone over $S_k.$ Then for each $k$ we have $\mathit{IC}_{S_k}=\mathbb Q_{S_k}[2k-1]$ and $\mathit{IC}_{X_k}=\bb Q_X[2k].$
\end{theorem}
\begin{proof}
	The result obviously holds for $S_1$ since $S_1\cong C$ is smooth. Now assume that the result holds for each $m=1,\ldots, k-1.$ By \Cref{cor:Nbhd}, $S_k$ is locally isomorphic to the product of $X_{k-1}$ (for a curve of smaller degree) with a smooth variety, hence the result holds for $S_k.$ By \Cref{cor:P1IH}, we therefore have
	\begin{align*}
		H^j(S_k,\bb Q) = \mathit{IH}^j(S_k) = 
		\begin{cases}
			\bb Q & j=0,2,4,\ldots, 4k-2,\\
			0 & \text{otherwise.}
		\end{cases}
	\end{align*}
	We similarly have the result at every point of $X_k$ away from the origin as well. Thus we just need to show that the natural map $\bb Q_{X_k}[2k]\to \mathit{IC}_{X_k}$ is an isomorphism at the origin.\\
	\indent For this, we blow up the origin of $\bb C^{2n+1},$ which yields the total space of the line bundle $\mathcal O(1)\to\bb P^{2n}.$ We get a diagram like so.
	\[
	\begin{tikzcd}
		{\bb P^{2n}} & {\bl_0\bb C^{2n+1}} & {\bb P^{2n}} \\
		{\{0\}} & {\bb C^{2n+1}}
		\arrow["i"', hook, from=1-1, to=1-2]
		\arrow["{\id}", bend left, from=1-1, to=1-3]
		\arrow[from=1-1, to=2-1]
		\arrow["p"', from=1-2, to=1-3]
		\arrow["\varepsilon", from=1-2, to=2-2]
		\arrow["{i_0}"', from=2-1, to=2-2]
	\end{tikzcd}
	\]
	For a pure perverse sheaf $K$ of geometric origin on $\bb P^{2n+1},$ the shifted pullback $p^*K[1]$ is again perverse and pure. This means that we can apply the decomposition theorem to get
	\[
	\varepsilon_*p^*K[1] \cong \widetilde K\oplus \bigoplus_{j\in \bb Z} H_j[-j],
	\]
	where $\widetilde K$ agrees with $p^*K[1]$ away from the origin and the $H_j$ are supported at the origin. We also have the relative hard Lefschetz isomorphisms $L^j:H_{-j}\to H_j$ and by base change we have isomorphisms
	\begin{align*}
		\pH^{j} i_0^*\widetilde K\oplus H_{j}\cong H^{j}(\bb P^{2n}, i^*p^*K[1])\cong H^{j+1}(\bb P^{2n}, K).
	\end{align*}
	Combining all of this, we get a diagram.
	\[
	\begin{tikzcd}
		{\pH^{-j} i_0^*\widetilde K\oplus H_{-j}} & { \pH^{j} i_0^*\widetilde K\oplus H_{j}} \\
		{H^{-j+1}(\bb P^{2n}, K)} & {H^{j+1}(\bb P^{2n}, K)}
		\arrow["{L^j}", from=1-1, to=1-2]
		\arrow["\cong"', from=1-1, to=2-1]
		\arrow["\cong", from=1-2, to=2-2]
		\arrow["{L^j}"', from=2-1, to=2-2]
	\end{tikzcd}
	\]
	By hard Lefschetz for $H^*(\bb P^{2n}, K)$ the bottom map is surjective and the kernel is by definition the primitive cohomology $H^{-j+1}_{\prim}(\bb P^{2n}, K).$ It's a general fact that $\pH^{j} i_0^*\widetilde K=0$ for $j\ge 1,$ so it follows that $\pH^{-j} i_0^*\widetilde K\cong H^{-j+1}_{\prim}(\bb P^{2n}, K)$ for $j\ge 1.$\\
	\indent Applying this to $K=\mathit{IC}_{S_k}\cong \bb Q_{S_k}[2k-1],$ we find that $\widetilde K\cong \mathit{IC}_{X_k}$ and
	\[
	\pH^{-j} i_0^*\mathit{IC}_{X_k}\cong H^{-j+1}_{\prim}(\bb P^{2n}, \mathit{IC}_{S_k})\cong H^{-j+2k}_{\prim}(S_k, \bb Q).
	\]
	Since the cohomology of $S_k$ is isomorphic to the cohomology of $\bb P^{2k-1},$ it must be generated in $H^0$ using the hard Lefschetz map. In particular, the only primitive cohomology is the one in $H^0.$ Similarly, when $K=\bb Q_{\bb P^{2n}}[2n],$ we have $\widetilde K \cong \bb Q_{\bb C^{2n+1}}[2n+1].$ Applying the same argument as above and shifting we get
	\[
	\pH^{-j} i_0^*\bb Q_{\bb C^{2n+1}}[2k-1]\cong H^{-j+2k}_{\prim}(\bb P^{2n}, \bb Q).
	\]
	We thus have a commutative diagram
	\[
	\begin{tikzcd}
		{\pH^{-j} i_0^*\bb Q_{\bb C^{2n+1}}[2k-1]} & {\pH^{-j} i_0^*\mathit{IC}_{X_k}} \\
		{H^{-j+2k}_{\prim}(\bb P^{2n}, \bb Q)} & {H^{-j+2k}_{\prim}(S_k, \bb Q)}
		\arrow[from=1-1, to=1-2]
		\arrow[Rightarrow, no head, from=1-1, to=2-1]
		\arrow[Rightarrow, no head, from=1-2, to=2-2]
		\arrow[from=2-1, to=2-2]
	\end{tikzcd}
	\]
	where the horizontal maps are the restriction maps. Since $S_k$ has the cohomology of projective space, this square is nonzero only for $j=2k,$ in which case the bottom map is clearly an isomorphism. Thus the top map is an isomorphism as well, meaning that $\mathit{IC}_{X_{k}}$ is isomorphic to $\bb Q_{X_k}$ at the origin.
\end{proof}

\subsection{Ordered partitions}\label{subsec:OrdPart}
Our computation of the cohomology of $F_n$ relies on stratifying $\bb C^{2n+1}$ in a particular way which we will describe in \Cref{subsec:StratifyCn}. However, we first need some elementary preliminaries. For the moment, let $n$ be an arbitrary positive integer. An \textbf{ordered partition} $P$ of $n$ is a tuple of positive integers
\[
    P=(p_1,\ldots,p_\ell)
\]
such that $p_1+\cdots +p_\ell = n.$ We call $\ell$ the \textbf{length} of $P$ and denote it by $|P|=\ell.$ We write $\gcd(P)$ in place of $\gcd(p_1,\ldots,p_\ell).$ We collect some facts about ordered partitions here.
\begin{facts}\label{fct:Partitions}
\phantom{3}\newline\vspace{-0.3cm}
\begin{enumerate}
    \item The set of ordered partitions of $n$ is in bijection with the powerset of $\{1,\ldots, n-1\}.$ It follows that the ordered partitions of $n$ of length $\ell$ are in bijection with the subsets of $\{1,\ldots,n-1\}$ of size $n-\ell.$ In particular, there are $2^{n-1}$ ordered partitions of $n$ and $\binom{n-1}{n-\ell}=\binom{n-1}{\ell-1}$ ordered partitions of $n$ with length $\ell.$

    \item If $\gcd(P)=d\neq 1$ then $d$ divides $n$ and $P=d\cdot Q$ where $Q=(p_1/d,\ldots,p_\ell/d)$ is an ordered partition of $n/d$ with $\gcd(Q)=1.$
    
    \item Let $g(n)$ be the number of ordered partitions $P$ of $n$ with $\gcd(P)=1.$ Then by the previous two facts we have
    \[
        \sum_{d\mid n} g(d) = 2^{n-1}.
    \]
    Therefore, by M\"obius inversion,
    \[
        g(n)=\sum_{d\mid n} \mu\left(\frac{n}{d}\right)2^{d-1}.
    \]
    If $g_\ell(n)$ denotes the number of ordered partitions $P$ of $n$ with $|P|=\ell$ and $\gcd(P)=1$ then we similarly have
    \[
        g_{\ell}(n) = \sum_{d\mid n} \mu\left(\frac{n}{d}\right)\binom{d-1}{\ell-1}.
    \]
\end{enumerate}
\end{facts}

\subsection{Stratifying affine space}\label{subsec:StratifyCn}
Now fix a positive integer $n.$ Recall that $f=\det H_n$ is the general Hankel determinant on $\bb C^{2n+1}$ whose zero locus is $X_n$ and whose affine Milnor fiber is $F_n.$ We will use the local structure of $F_n$ to compute its cohomology, but it will be convenient to stratify the whole of $\bb C^{2n+1}.$ The strata will be denoted by $Y_P$ and $Y_{P,0}$ where $P$ ranges over ordered partitions of $n+1.$\\
\indent We construct this stratification inductively. In the base case $n=0$ we need to stratify $\bb C.$ Call the coordinate $x_0.$ There is only a single partition of $1$ and we set
\begin{align*}
    Y_{(1)} &= \{x_0\neq 0\}=\bb C^*,\\
    Y_{(1),0} &= \{x_0=0\}=\{0\}.
\end{align*}
For arbitrary $n\ge 1,$ we first set
\[
        Y_k=\{x\in \bb C^{2n+1}\mid x_j=0 \text{ for }j\le k-1\text{ and }x_k\neq 0\}
\]
for $k=0,\ldots, n.$ The coordinates $y_0,\ldots, y_{2n-k}$ from \Cref{lem:CoordChange} give us an isomorphism
\begin{equation}\label{eqn:YAsAProduct}
    Y_k\cong \bb C^*\times \bb C^{k+1}\times \bb C^{2n-2k-1}
\end{equation}
where $y_0$ is the coordinate on the first factor, $y_1,\ldots, y_{k+1}$ are the coordinates on the second factor, and $y_{k+2},\ldots, y_{2n-k}$ are the coordinates on the third factor. By induction, for each $k$ we have a stratification of the third factor $\bb C^{2n-2k-1}$ whose strata we denote by $Z_Q$ and $Z_{Q,0},$ indexed by the ordered partitions $Q$ of $n-k.$ This induces strata $\bb C^*\times \bb C^{k+1}\times Z_Q$ and $\bb C^*\times \bb C^{k+1}\times Z_{Q,0}$ on $Y_k.$ If $P=(k+1,p_2,\ldots,p_\ell)$ is an ordered partition of $n+1,$ then $Q=(p_2,\ldots,p_\ell)$ is an ordered partition of $n-k$ and we set
\begin{align*}
    Y_P &= \bb C^*\times \bb C^{k+1}\times Z_Q,\\
    Y_{P,0} &= \bb C^*\times \bb C^{k+1}\times Z_{Q,0}.
\end{align*}
This constructs strata $Y_P$ and $Y_{P,0}$ of $\bb C^{2n+1}$ for each ordered partition $P$ of $n+1.$\\
\indent The reason for introducing this stratification is contained in the following.
\begin{proposition}\label{prop:StratFacts}
    Let $P=(p_1,\ldots, p_\ell)$ be an ordered partition of $n+1$ and let $Y_P$ and $Y_{P,0}$ be the corresponding strata of $\bb C^{2n+1}.$ Then the function $f$ vanishes identically on $Y_{P,0}$ and is nonvanishing on $Y_P.$ In particular, $X_n = \bigcup_P Y_{P,0}.$ Furthermore, there are coordinates $y_i$ on $Y_P$ which induce an isomorphism $Y_P\cong (\bb C^*)^{\ell}\times \bb C^n$ and in these coordinates we have
    \[
        f|_{Y_P} = y_0^{p_1}\cdots y_\ell^{p_\ell}.
    \]
\end{proposition}
\begin{remark}
    By an abuse of notation we are using the symbols $y_i$ for coordiantes on $Y_P,$ but these are not the same as the coordiantes on any $Y_k$ in \Cref{lem:CoordChange} or in the construction of the stratification above.
\end{remark}
\begin{proof}
    We go by induction. The claim is clear in the case $n=0$ by the construction of the stratification, since in this case $f=x_0.$ If $n\ge 1,$ write $Q=(p_2,\ldots, p_\ell).$ By \Cref{lem:CoordChange} there are coordinates $y_i$ on $Y_{p_1-1}$ which induce the isomorphism in (\ref{eqn:YAsAProduct}) and in these coordinates we have
    \[
        f|_{Y_{p_1-1}}=y_0^{p_1}\det H_{n-k-1}(y_{k+2},\ldots,y_{2n-k}).
    \]
    By induction, $\det H_{n-k-1}$ vanishes identically on $Z_{Q,0},$ so $f$ vanishes identically on $Y_{P,0}.$ Furthermore, by induction there are coordinates $z_i$ on $Z_Q\subset \bb C^{2n-2k-1}$ which induce an isomorphism $Z_Q\cong (\bb C^*)^{\ell-1}\times\bb C^{n-k-1}$ such that in these coordinates
    \[
        \det H_{n-k-1}= z_{q_2}^{p_2}\cdots z_{q_{\ell}}^{p_{\ell}}
    \]
    for some $q_2,\ldots, q_{\ell}.$ It follows after relabeling the coordinates $(y_0,\ldots, y_k, z_0,z_1,\ldots)$ we have the desired expression for $f$ and we get an isomorphism
    \[
        Y_P = \bb C^*\times \bb C^{k+1}\times Z_P \cong (\bb C^*)^\ell\times \bb C^n.
    \]
\end{proof}
\begin{corollary}\label{cor:SimpleF}
    In the setting of \Cref{prop:StratFacts}, let $d=\gcd(P).$ Then we can change coordinates on $Y_P\cong (\bb C^*)^\ell\times \bb C^n$ such that $f|_{Y_P}=z^d,$ where $z$ is a coordinate on one of the $\bb C^*$ factors.
\end{corollary}
\begin{proof}
    Consider the monomial $x^ay^b$ on $(\bb C^*)^2$ where $a>b$ and $\gcd(a,b)=1.$ Write $a=qb+r$ with $q\ge 0$ and $0\le r< b$ so that $x^ay^b = x^r(x^qy)^b = x_1^ry_1^b$ where $x_1=x$ and $y_1=x^qy$ form a coordinate system on $\bb C^*.$ Continuing in this way, the Euclidean algorithm guarantees that we will end up with coordinates $x_k,y_k$ such that either $x^ay^b = x_k^d$ or $x^ay^b=y_k^d.$ Performing this procedure repeatedly to pairs of factors in the product
    \[
        f|_{Y_P} = y_0^{p_1}\cdots y_{\ell}^{p_\ell}
    \]
    yields a coordinate system on $Y_P$ with $f=z^d$ for some coordinate $z.$
\end{proof}

\indent Let $P={p_1,\ldots, p_\ell}$ be a partition of $n+1.$ The proof of \Cref{lem:CoordChange} shows that, after restricting to $Y_{p_1-1},$ we can think of the Hankel matrix $H_n$ as being the same (for the purposes of the hypersurface defined by $\det H_n = 0$) as the matrix in (\ref{eqn:NewHankel}). Repeating this procedure for the lower right block and continuing in this way, we find that the stratum $Y_P$ corresponds to a way of ``turning $H_n$ into a block diagonal matrix'' and the coordinate change functions to make each block a ``skew lower triangular'' matrix:
\begin{align}\label{eqn:BMatrix}
\begin{pNiceMatrix}
    x_0 & x_1 & \cdots & x_{n-1}\\
    x_1 & x_2 & \cdots & x_n\\
    \vdots & \vdots & \ddots & \vdots\\
    x_{n-1} & x_n & \cdots & x_{2n}
\end{pNiceMatrix}
\rightsquigarrow
\begin{pNiceMatrix}
    P_1 & 0 & \cdots & 0\\
    0 & P_2 & \cdots & 0\\
    \vdots & \vdots & \ddots & \vdots\\
    0 & 0 & \cdots & P_\ell
\end{pNiceMatrix},
\end{align}
where the $P_i$ are Hankel matrices of size $p_i\times p_i$ in which the entries above the main skew diagonal are all zero.
\begin{align}
    P_i=
    \begin{pNiceMatrix}
        0 & \cdots & 0 & y_{q_i}\\
        0 & \cdots & y_{q_i} & y_{q_i+1}\\
        \vdots & \iddots & \vdots & \vdots\\
        y_{q_i} & \cdots & y_{q_i+p_i-1} & y_{q_i+p_i}
    \end{pNiceMatrix}
\end{align}
This description, while not entirely rigorous, perhaps provides an intuitive picture for the strata $Y_P$ and the form that $f$ takes on each one.


\begin{example}[$n=2$]\label{ex:N2Strata}
    In this case we work on $\bb C^5$ and our matrix is
    \[
    \begin{pNiceMatrix}
        x_0 & x_1 & x_2\\
        x_1 & x_2 & x_3\\
        x_2 & x_3 & x_4
    \end{pNiceMatrix}.
    \]
    We have $4$ strata corresponding to the $4$ ordered partitions of $3.$
    \begin{itemize}
        \item $P=(1,1,1)$ corresponds to the block diagonal matrix
        \[
        \begin{pNiceMatrix}
            \Block[borders={bottom,right}]{1-1}{y_0} & 0 & 0\\
            0 & \Block[borders={top,bottom,left,right}]{1-1}{y_2} & 0\\
            0 & 0 & \Block[borders={top,left}]{1-1}{y_4}
        \end{pNiceMatrix}
        \]
        and on $Y_P\cong (\bb C^*)^3\times \bb C^2$ we have $f|_{Y_P} = y_0y_2y_4.$
        \item $P=(2,1)$ corresponds to the block diagonal matrix
        \[
        \begin{pNiceMatrix}
            \Block[borders={bottom,right}]{2-2}{}0 & y_1 & 0\\
            y_1 & y_2 & 0\\
            0 & 0 & \Block[borders={top,left}]{1-1}{y_4}
        \end{pNiceMatrix}
        \]
        and on $Y_P\cong (\bb C^*)^2\times \bb C^2$ we have $f|_{Y_P} = y_1^2y_4.$
        \item $P=(1,2)$ corresponds to the block diagonal matrix
        \[
        \begin{pNiceMatrix}
            \Block[borders={bottom,right}]{1-1}{y_0} & 0 & 0\\
            0 & \Block[borders={top,left}]{2-2}{}0 & y_3\\
            0 & y_3 & y_4
        \end{pNiceMatrix}
        \]
        and on $Y_P\cong (\bb C^*)^2\times \bb C^2$ we have $f|_{Y_P} = y_0y_3^2.$
        \item $P=(3)$ corresponds to the block diagonal matrix
        \[
        \begin{pNiceMatrix}
            0 & 0 & y_2\\
            0 & y_2 & y_3\\
            y_2 & y_3 & y_4
        \end{pNiceMatrix}
        \]
        and on $Y_P\cong (\bb C^*)^3\times \bb C^2$ we have $f|_{Y_P} = y_2^3.$
    \end{itemize}
\end{example}

\subsection{The Hodge polynomial}\label{subsec:HodgePoly}
Here we give a very brief review of the Hodge polynomial. For a more detailed introduction see \cite{Sri02}. The main theorem we need is the following.
\begin{theorem}\label{thm:HodgePoly}
    There is a unique way to assign to each complex algebraic variety $X$ a polynomial $h_X(u,v)$ with integer coefficients such that
    \begin{enumerate}
        \item if $X$ is smooth and projective, then
            \[
                h_X(u,v) = \sum_{p,q\ge 0} h^{p,q}(X)u^pv^q
            \]
        where $h^{p,q}(X)=\dim H^q(X,\Omega_X^p),$
        \item if $Z\subset X$ is closed and $U=X\setminus Z,$ then
            \[
                h_Y(u,v) = h_Z(u,v) + h_U(u,v),
            \]
        \item if $E\to X$ is a Zariski locally trivial fiber bundle with fiber $F$ (in particular if $E=X\times F$) then
            \[
                h_E(u,v)=h_X(u,v)h_F(u,v).
            \]
    \end{enumerate}
\end{theorem}
We call $h_X(u,v)$ the \textbf{Hodge polynomial} of $X.$ For example, if $X$ is a union of $d$ points, then $h_X(u,v)=d.$ We also have
\begin{align*}
	h_{\bb P^n}(u,v) &= 1+uv+\cdots+u^nv^n,\\
	h_{\bb C^n}(u,v) &= u^nv^n,\\
	h_{\bb C^*}(u,v) &= -1+uv.
\end{align*}

For arbitrary $X$ it is not true that the coefficients of the Hodge polynomial $h_X$ are the dimensions of the cohomology of $X.$ However, there is a general formula in terms of the compactly supported cohomology of $X.$
\begin{theorem}\label{thm:HodgePolyFormula}
    For each complex algebraic variety $X,$ let
    \begin{align}
        h_X(u,v) = \sum_{p,q,i\ge 0}(-1)^{i}\dim\gr_F^p\gr_{p+q}^W H^i_c(X,\bb C)u^pv^q.
    \end{align}
    Then this assignment satisfies the properties in \Cref{thm:HodgePoly}.
\end{theorem}
\begin{corollary}\label{cor:PureHodgePoly}
    If $X$ is an algebraic variety such that each $H^i_c(X,\bb Q)$ is pure of weight $i,$ and $h^{p,q}$ are defined so that
    \[
        h_X(u,v) = \sum_{p,q\ge 0} h^{p,q}u^pv^q,
    \]
    then $\dim H^i_c(X,\bb C) = (-1)^i\sum_{p,q=i} h^{p,q}.$
\end{corollary}

\subsection{The cohomology of the Milnor Fiber}\label{subsec:MilFib}
\indent We will compute the cohomology of $F_n$ in two parts. First, we will compute the Hodge polynomial of $F_n,$ then we will show that each $H^i(F_n,\bb Q)$ is a pure Hodge structure of weight $i.$\\
\begin{proposition}\label{prop:FnHodgePoly}
    If $n\ge 1$ and $F_n$ is the Milnor fiber for $f=\det H_n,$ then the Hodge polynomial of $F_n$ is
    \[
        h_{F_n}(u,v) = (uv)^{n-1}\sum_{d|(n+1)} \varphi\left(\frac{n+1}{d}\right)(uv)^d,
    \]
    where $\varphi(k)=|\bb Z/k\bb Z^\times|$ is the Euler function.
\end{proposition}
\begin{proof}
    To make the formulas a bit nicer, we will prove the proposition for $F_{n-1}.$ Stratify $\bb C^{2n-1}$ as in the discussion in \S\ref{subsec:StratifyCn}. This induces a stratification of $F_{n-1}$ by $F_{n-1}\cap Y_P.$ The $Y_{P,0}$ do not appear since $F_{n-1}\cap Y_{P,0}=\varnothing$ for each $P$ by \Cref{prop:StratFacts}.\\
    \indent By the same proposition and \Cref{cor:SimpleF}, for each ordered partition $P=(p_1,\ldots, p_\ell)$ of $n,$ the closed subset $F_{n-1}\cap Y_P$ is the set in $Y_P\cong (\bb C^*)^{\ell}\times (\bb C)^{n-1}$ on which $f=z^d=1,$ where $z$ is a coordinate on one of the $\bb C^*$ factors and $d=\gcd(P).$ It follows that $F_{n-1}\cap Y_P$ is a product of $(\bb C^*)^{\ell-1}\times \bb C^{n-1}$ with a union of $d$ points, so by \Cref{thm:HodgePoly} and \Cref{ex:HodgePolys} we have
    \[
        h_{F_{n-1}\cap Y_P}(u,v) = \gcd(p_1,\ldots,p_\ell)(uv)^{n-1}(uv-1)^{\ell-1}.
    \]
    Since these polynomials only depend on the product $uv,$ write $t=uv.$ By (2) in \Cref{thm:HodgePoly}, $h_{F_{n-1}}(t)$ is the sum of these polynomials over all ordered partitions $P$ of $n.$
    \begin{align*}
        h_{F_{n-1}}(t) = \sum_{P}\gcd(P)t^{n-1}(t-1)^{|P|-1}
    \end{align*}
    Splitting up the sum based on the length and $\gcd$ of the partition $P$ yields
    \begin{align*}
        h_{F_{n-1}}(t) = \sum_{d\mid n}\sum_{|P|=\ell}\sum_{\gcd(P)=d}dt^{n-1}(t-1)^{\ell-1}
    \end{align*}
    Now recall that the number of ordered partitions $P$ of $n$ with $|P|=\ell$ and $\gcd(P)=d$ is the same as the number of ordered partitions $P$ of $\frac{n}{d}$ with $|P|=\ell$ and $\gcd(P)=1,$ which is the number $g_\ell\left(\frac{n}{d}\right).$ From \Cref{fct:Partitions} we have
    \[
        g_\ell\left(\frac{n}{d}\right) = \sum_{m\mid \frac{n}{d}} \mu\left(\frac{n/d}{m}\right)\binom{m-1}{\ell-1}.
    \]
    Hence we can write
    \begin{align*}
        h_{F_{n-1}}(t) &= t^{n-1}\sum_{d\mid n}\sum_{\ell=1}^n dg_{\ell}\left(\frac{n}{d}\right)(t-1)^{\ell-1}\\
        &= t^{n-1}\sum_{d\mid n}\sum_{\ell=1}^n \sum_{m\mid \frac{n}{d}}d\mu\left(\frac{n/d}{m}\right)\binom{m-1}{\ell-1}(t-1)^{\ell-1}\\
        &= t^{n-1}\sum_{d\mid n}\sum_{m\mid \frac{n}{d}}d\mu\left(\frac{n/d}{m}\right)\sum_{\ell=1}^m\binom{m-1}{\ell-1}(t-1)^{\ell-1}\\
        &= t^{n-1}\sum_{d\mid n}\sum_{m\mid \frac{n}{d}}d\mu\left(\frac{n/d}{m}\right)t^{m-1}.
    \end{align*}
    The third equality is true since the binomial coefficients are zero if $\ell>m$ and the last equality is the binomial theorem applied to $t^{m-1}=((t-1)+1)^{m-1}.$ Finally, observe that $d\mid n$ and $m\mid \frac{n}{d}$ if and only if $m\mid n$ and $d\mid \frac{n}{m}.$ Therefore we can switch the sums to isolate the coefficient of $t^{m-1}$ and get
    \begin{align*}
        h_{F_{n-1}}(t) &= t^{n-1}\sum_{d\mid n}\sum_{m\mid \frac{n}{d}}d\left(\frac{n/d}{m}\right)t^{m-1}\\
        &= t^{n-1}\sum_{m\mid n}t^{m-1}\sum_{d\mid \frac{n}{m}}d\mu\left(\frac{n/d}{m}\right)\\
        &= t^{n-1}\sum_{m\mid n} \varphi\left(\frac{n}{m}\right)t^{m-1},
    \end{align*}
    which the desired polynomial.
\end{proof}

\indent Now we just need to show that each $H^i_c(F_n)$ is pure of weight $i.$
\begin{proposition}\label{prop:PureMilFib}
    Each $H^i_c(F_n,\bb Q)$ is a pure Hodge structure of weight $i.$
\end{proposition}
\begin{proof}
    Let $\bb P^{2n+1}$ have coordinates $x_0,\ldots,x_{2n},y$ and let $X\subset \bb P^{2n+1}$ be the zero locus of the function $g(x,y)=f(x)-y^{n+1}.$ By setting $y=0$ we see that $S_n$ naturally is a closed subset of $X.$ Write $\iota:S_n\hookrightarrow X$ for the inclusion. The complement is obtained by setting $y=1$ and we see that this is the affine Milnor fiber $F_n.$ So $X=F_n\sqcup S_n$ can be seen as a disjoint union. We have an exact sequence in the cohomology in which the restriction map $\iota^*:H^i(X,\bb Q)\to H^i(S_n,\bb Q)$ commutes with the restriction map from projective space.
    \[
    \begin{tikzcd}
        \cdots & {H^i_c(F_n,\bb Q)} & {H^i(X,\bb Q)} & {H^i(S_n,\bb Q)} & \cdots \\
        && {H^i(\bb P^{2n+1},\bb Q)} & {H^i(\bb P^{2n},\bb Q)} &
        \arrow[from=1-1, to=1-2]
        \arrow[from=1-2, to=1-3]
        \arrow["\iota^*", from=1-3, to=1-4]
        \arrow[from=1-4, to=1-5]
        \arrow[from=2-3, to=1-3]
        \arrow[from=2-3, to=2-4]
        \arrow["\cong"', from=2-4, to=1-4]
    \end{tikzcd}
    \]
    By \Cref{cor:P1IH} and \Cref{thm:PureQSec}, the restriction map $H^i(\bb P^{2n},\bb Q)\to H^i(S_n,\bb Q)$ is an isomorphism for each $i\le 2n-1.$ It follows that the map $\iota^*$ must be surjective and we get a splitting
    \begin{equation}
        H^i(X,\bb Q) \cong H^i_c(F_n,\bb Q)\oplus H^i(S_n,\bb Q).
    \end{equation}
    Therefore, to get purity of $H^i_c(F_n,\bb Q),$ it suffices to show that $H^i(X,\bb Q)$ is pure.\\
    \indent Recall that $g(x,y)=f(x)-y^{n+1}$ is the defining equation for $X$ in $\bb P^{2n+1}.$ By \Cref{lem:CoordChange}, we can cover $\bb P^{2n+1}$ by affine opens $U$ on which $f$ looks like the determinant of a smaller Hankel matrix. For such affine opens we have a distinguished triangle
    \[
        \bb Q_X[2n]|_U \to \psi_{g,1}\bb Q_{\bb P^{2n+1}}[2n+1]|_U \to \varphi_{g,1}\bb Q_{\bb P^{2n+1}}[2n+1]|_U\to\cdots.
    \]
    Then by Thom-Sebastiani \cite[Theorem 10.3.16]{Max19} we have an isomorphism
    \[
        \varphi_{g,1}\bb Q_{\bb P^{2n+1}}[2n+1]|_U\cong \sum_{\alpha\beta=1} \varphi_{f,\alpha}\bb Q_{\bb C^{2n}}[2n]\otimes \varphi_{y^{n+1},\beta}\bb Q_{\bb C}[1],
    \]
    which respects the monodromy. Since $f$ is the determinant of a smaller Hankel matrix on $U,$ by induction we can say that both factors of each summand on the right hand side are pure, hence the left hand side is also pure. If $N$ is the nilpotent operator on the vanishing cycles, then by \Cref{cor:N2=0} this means that that $N=0$ on the right hand side, so it is true on the left. Since $\varphi_{g,1}\bb Q_{\bb P^{2n+1}}[2n+1]=\im N,$ this means that $N^2=0$ on $\psi_{g,1}\bb Q_{\bb P^{2n+1}}[2n+1].$ It follows that the monodromy weight filtration on $\psi_{g,1}\bb Q_{\bb P^{2n+1}}[2n+1]$ lives in weights $2n+1,$ $2n,$ and $2n-1.$ Therefore $\bb Q_X[2n]$ only has weights $2n$ and $2n-1.$ Explicitly, we have a distinguished triangle
    \begin{equation}\label{eqn:CoverTriangle}
    \begin{tikzcd}
        K\ar[r]& \bb Q_X[2n]\ar[r]& \mathit{IC}_X\ar[r,"+1"]& \cdots
    \end{tikzcd}
    \end{equation}
    where $K$ is pure of weight $2n-1.$ This yields the diagram
    \begin{equation}\label{eqn:CoverExSeq}
    	\begin{tikzcd}
    		\cdots\ar[r]& H^{-i}(K)\ar[r,"j"]\ar[d,"L^i"]& H^{2n-i}(X)\ar[r,"q"]\ar[d,"L^i"]& \mathit{IH}^{2n-i}(X)\ar[r]\ar[d,"L^i"]& \cdots\\
    		\cdots\ar[r]& H^i(K)\ar[r,"j"]& H^{2n+i}(X)\ar[r,"q"]& \mathit{IH}^{2n+i}(X)\ar[r]& \cdots 
    	\end{tikzcd}
    \end{equation}
    where the rows are the long exact sequence coming from (\ref{eqn:CoverTriangle}) and the vertical maps $L^i$ are the hard Lefschetz maps. Note that these maps are isomorphisms on the left and right terms since $K$ and $\mathit{IC}_X$ are pure.\\
    \indent From this we can show that the cohomology of $X$ is pure. To do this, we need to show that the map $j$ is zero for all $i.$ We have commutative diagrams 
    \[
    \begin{tikzcd}
        H^{2n-i}(X)\ar[r,"q"]& \mathit{IH}^{2n-i}(X)\\
        H^{2n-i}(\bb P^{2n+1})\ar[u,"r"]\ar[ur]
    \end{tikzcd}
    \]
    where $r$ is the map induced by the inclusion $X\to\bb P^{2n+1}.$ To show that $j=0$ It suffices to show that $q$ is injective for all $i.$ Observe that $q\circ r$ is an isomorphism in degree $0.$ By repeatedly applying $L$ to this case, it follows that $q\circ r$ must be injective for all $i$ (here we use the fact that $\bb P^{2n+1}$ has no odd-degree cohomology). By the Lefschetz hyperplane theorem, $r$ is an isomorphism for $i\ge 1$ and so it follows that $q$ must be injective for $i\ge 1.$ Thus $j=0$ in the negative degrees. Since $L^i$ is an isomorphism on the cohomology of $K,$ it follows that $j=0$ in the positive degrees as well.\\
    \indent To get purity of the middle cohomology, we use the Hodge polynomial. By \Cref{prop:FnHodgePoly} we have
    \[
        h_X(u,v) = h_{F_n}(u,v) + h_{S_n}(u,v).
    \]
    Since both terms on the right hand side only contain even degree monomials of the form $u^pv^p,$ the same is true of $h_X.$ However, $H^{2n}(X)$ only has weights $2n$ and $2n-1$ and the rest of the cohomology is pure. Applying \Cref{thm:HodgePolyFormula} we find that the coefficients of the odd degree monomial $u^pv^{2n-1-p}$ are the numbers
    \[
        \dim \gr_F^p\gr^W_{2n-1} H^{2n}(X) - \dim \gr_F^p H^{2n-1}(X)=0.
    \]
    The odd degree cohomology is zero by the Lefschetz hyperplane theorem, so it follows that the graded pieces of $W_{2n-1}H^{2n}(X)$ are zero as well. Therefore $W_{2n-1}H^{2n}(X)=0,$ completing the proof.
\end{proof}

\indent It immediately follows from \Cref{cor:PureHodgePoly} that the dimensions of the $H^i_c(F_n)$ are given by the coefficients of the Hodge polynomial of $F_n.$ Applying Poincar\'e duality gives the following.
\begin{corollary}\label{cor:MilFibCoho}
    The cohomology of $F_n$ is pure and of Hodge-Tate type, and the dimensions are given by
    \[
        \dim H^i(F_n,\bb C) =
        \begin{cases}
            \varphi\left(\frac{n+1}{d}\right) & i=n+1-d \text{ where } d\mid(n+1),\\
            0 & \text{otherwise}.
        \end{cases}
    \]
\end{corollary}

\subsection{Eigenvalues of the monodromy action}\label{subsec:Mono}
The formula in \Cref{cor:MilFibCoho} suggests that the eigenspaces for the monodromy operator $T:H^*(F_n,\bb C)\to H^*(F_n,\bb C)$ correspond to primitive roots of unity, with each $(n+1)/d$-th primitive root having a $1$-dimensional eigenspace in $H^{n+1-d}.$ We now show that this is actually the case.
\begin{proposition}\label{prop:Mono}
    For all divisors $d$ of $n+1$ and all primitive $(n+1)/d$-th roots of unity $\lambda,$ the $\lambda$-eigenspace of the monodromy operator $T$ is $1$-dimensional and lies in $H^{n+1-d}(F_n,\bb C).$
\end{proposition}
\begin{proof}
    Since the action on $F$ is given by multiplying by $(n+1)$-th roots of unity, the eigenvalues of the induced action on $H^*(F_n,\bb C)$ can only be roots of unity. So we only need to find the eigenspaces for these eigenvalues.\\
    \indent Recall that for each $k\in\bb N$ the symbol $\mu_{k}$ denotes the group of $k$-th roots of unity, whose action on $F_n$ is the monodromy action $T.$ For each divisor $d$ of $n+1$ let $F_{n,d}=F_n/(\mu_{(n+1)/d}).$ Since $F_n$ is smooth and the action of $\mu_{n+1}$ is free, the cohomology of $F_{n,d}$ is the part fixed by $\mu_{(n+1/d)}.$
    \[
        H^k(F_{n,d},\bb C) = H^k(F_n,\bb C)^{\mu_{(n+1)/d}}
    \]
    If $\lambda$ is a generator of $\mu_{n+1},$ then $\lambda^d$ generates $\mu_{(n+1)/d},$ so $H^k(F_{n,d},\bb C)$ is the part fixed by $\lambda^d.$ This means exactly that each $H^k(F_{n,d},\bb C)$ is the sum of the eigenspaces with eigenvalues which are $d$-th roots of unity. The cohomology of $F_{n,d}$ is pure since it is a sub Hodge structure of $H^k(F_n,\bb Q)$ which is pure. So by \Cref{cor:PureHodgePoly}, to compute the dimensions it suffices to compute the Hodge polynomial of $F_{n,d}.$ In order to do this, we will find a convenient $\bb C^*$-bundle on $F_{n,d}$ whose Hodge polynomial can be computed.\\
    \indent Let $\bb C^{2n+2}$ be the affine space with coordinates $x_0,\ldots,x_{2n},y$ and define
    \[
        G_{n,d}=\{(x,y)\in\bb C^{2n+2}\mid y^df(x)=1\}.
    \]
    We have a natural map
    \[
        p:G_{n,d}\to F_{n,d}
    \]
    given by $p(x,y)=[y^{d/(n+1)}x].$ Note that $y^{d/(n+1)}x$ is well-defined only up to multiplication by a $(n+1)/d$-th root of unity, so the class in the quotient is well-defined. Now define a $\bb C^*$-action on $G_{n,d}$ by the formula
    \[
        s\cdot(x,y)=(s^{-1}x,s^{(n+1)/d}y)
    \]
    for $s\in\bb C^*.$ This action gives $G_{n,d}$ the structure of a $\bb C^*$-torsor over $F_{n,d}$ via the map $p.$ If we pull $p$ back by the quotient map $q:F_n\to F_{n,d}$ we get a trivial $\bb C^*$-bundle. It follows that $G_{n,d}$ is an \'etale locally trivial $\bb C^*$-bundle over $F_{n,d},$ hence it is Zariski locally trivial (because $\mathit{GL_n(\bb C)}$ is ``special'').
    \begin{equation}
    \begin{tikzcd}
        F_n\times \bb C^*\ar[r,"p'"]\ar[d,"q'"']& F_n\ar[d,"q"]\\
        G_{n,d}\ar[r,"p"']& F_{n,d}
    \end{tikzcd}
    \end{equation}
    Here $p'$ is projection onto the first factor and $q'(x,t)=(t^{-1}x,t^{(n+1)/d}).$\\
    \indent The function defining $G_{n,d}$ is similar enough to $f$ that computing its Hodge polynomial is doable in the same way. We partition $\bb C^{2n+2}$ in exactly the same way as in \S\ref{subsec:StratifyCn} so that $y^{d}f(x)$ is a product of monomials on each stratum after some coordinate change. These strata again correspond to ordered partitions of $n+1.$ If $P=(p_1,\ldots,p_\ell)$ is a partition and $Z_P$ is a stratum, then just as in the proof of \Cref{prop:FnHodgePoly} we have
    \[
        h_{G_{n,d}\cap Z_P}(u,v) = \gcd(d,p_1,\ldots,p_\ell)(uv)^{n}(uv-1)^{\ell-1}.
    \]
    Note that we now have this extra $d$ appearing in the $\gcd.$ Summing over all ordered partitions and simplifying the sum in the same way yields the formula
    \[
        h_{G_{n,d}}(u,v) = (uv)^{n-1}(uv-1)\sum_{\frac{n+1}{d}\mid m\mid (n+1)} (uv)^{m}\varphi\left(\frac{n+1}{m}\right).
    \]
    Since $G_{n,d}$ is a $\bb C^*$-bundle over $F_{n,d},$ applying (3) in \Cref{thm:HodgePoly} and \Cref{ex:HodgePolys}(4) gives
    \begin{equation}\label{eqn:FndHodgePoly}
        h_{F_{n,d}}(u,v) = (uv)^{n-1}\sum_{\frac{n+1}{d}\mid m\mid (n+1)} (uv)^{m}\varphi\left(\frac{n+1}{m}\right).
    \end{equation}
    \indent Now we can compare the coefficients of $h_{F_n}(u,v)$ and $h_{F_{n,d}}(u,v)$ to find that
    \begin{align}
        H^{n+m-1}_c(F_{n,d})\cong
        \begin{cases}
            H^{n+m-1}_c(F_n) & \frac{n+1}{d}\mid m,\\
            0 & \text{otherwise.}
        \end{cases}
    \end{align}
    Hence $H^{n+d-1}_c(F_n)$ is fixed by $\mu_{(n+1)/d},$ meaning that the eigenspaces contained in it are associated with $d$-th roots of unity. The same kind of argument shows that if $\lambda\in\mu_{m}$ with $m\mid d$ then no $\lambda$-eigenspace of $T$ is contained in $H^{n+d-1}_c(F_n).$ It follows that the only eigenspaces in $H^{n+d-1}_c(F_n)$ are those whose associated eigenvalues are primitive $d$-th roots of unity. By \Cref{cor:MilFibCoho}, each eigenspace has dimension at most $1,$ so each of their dimensions must be exactly $1.$ Applying Poincar\'e duality gives the result of the proposition.
\end{proof}

\subsection{The main theorem}\label{subsec:VanCycPf}
We now have the results needed to prove the main theorem: the computation of the nearby and vanishing cycles for the function $f=\det H_n.$ More precisely, we compute $\psi_f \bb Q_{\bb C^{2n+1}}[2n+1]$ and $\varphi_f \bb Q_{\bb C^{2n+1}}[2n+1],$ where we consider $f$ as a function on affine space $\bb C^{2n+1}.$ The theorem is as follows.
\begin{theorem}\label{thm:VanCyc}
	Let $f=\det H_n$ and let $X_n$ be as above.
	\begin{enumerate}
		\item All eigenvalues of $T:\psi_f \bb Q_{\bb C^{2n+1}}[2n+1]\to \psi_f \bb Q_{\bb C^{2n+1}}[2n+1]$ are of the form $\lambda = e^{2\pi i p/q}$ where $q\in\{1,\ldots,n+1\}$ and $\gcd(p,q)=1.$
		\item For each eigenvalue $\lambda$ of $T,$ the nearby cycle sheaf $\psi_{f,\lambda}\bb Q_{\bb C^{2n+1}}[2n+1]$ is pure of weight $2n.$
		\item If $\lambda = e^{2\pi i p/q}$ is an eigenvalue of $T$ with $q\neq 1,$ then $$\psi_{f,\lambda}\bb Q_{\bb C^{2n+1}}[2n+1]=\mathit{IC}(L_\lambda)$$ where $L_\lambda$ is a rank $1$ local system on $X_{n-q+1}.$
		\item $\varphi_{f,1}\bb Q_{\bb C^{2n+1}}[2n+1]=0,$ so $\psi_{f,1}\bb Q_{\bb C^{2n+1}}[2n+1] = \bb Q_{X}[2n].$
	\end{enumerate}
\end{theorem}
By \Cref{cor:Nbhd}, we can prove the theorem by induction. The difficult part is understanding what happens at the origin. However, this is taken care of by our work computing $H^*(F_n,\bb C).$
\begin{proof}
	By the same argument as in \Cref{cor:Nbhd}, each point $x\in X_n\setminus\{0\}$ has a neighborhood $U$ with $U\cong V\times X_{m}$ where $m < n$ and $V$ is smooth. Thus by induction (1) is true away from the origin. At the origin, (1) follows from the arguments given in \S \ref{subsec:Mono}.\\
	\indent Since $\bb Q_{X_n}[2n]$ is pure, $\psi_{f,1}\bb Q_{\bb C^{2n+1}}[2n+1]$ is as well. The weight filtration induced by the nilpotent operator $N$ is therefore trivial, which means that $N=0.$ Thus
	\[
		\varphi_{f,1}\bb Q_{\bb C^{2n+1}}[2n+1]=\im N=0.
	\]
	This proves (4).\\
	\indent Now we prove (2) and (3). Let $i_0:\{0\}\to X_n$ be the inclusion of the origin. Let $q\in\{2,\ldots, n+1\}$ and let $\lambda\in\mu_q$ be a primitive $q$-th root of unity. Let
	\[
		P_\lambda = \psi_{f,\lambda}\bb Q_{\bb C^{2n+1}}[2n+1].
	\]
	If $q=n+1$, then $P_\lambda$ is supported at the origin, and is just the $\lambda$-eigenspace of $T$ in the cohomology of $F_n,$ which has rank $1$ by the arguments in \S \ref{subsec:Mono}. If $q<n+1,$ then by induction $P_\lambda$ is pure of weight $2n$ away from the origin and we can write
	\[
	P_\lambda = P_\lambda' \oplus P_\lambda''
	\]
	where $P'_\lambda$ is supported on $X_{n-q+1}$ and $P_\lambda''$ is supported at $0.$ But the cohomology vector spaces $H^k(i_0^*P_\lambda)$ are the $\lambda$-eigenspaces of $T$ in $H^k(F_n,\bb Q).$ By \Cref{prop:Mono}, the $\lambda$-eigenspaces for $T$ occur in negative degree cohomology, so it cannot have a component supported at the origin. Thus $P_\lambda''=0$ in this case. The purity follows from the purity of the cohomology of each $F_k$ for $k\le n.$ This completes the proof.

	Suppose $q\neq n+1.$ Then each nonzero cohomology occurs in negative degree, so $P_\lambda''=0.$ Finally, $P_\lambda$ is pure of rank $1,$ and since the weight filtration is symmetric about weight $2n$ (see \Cref{prop:WeightFilt}), it must be of weight $2n.$ This proves the claims in (2) and (3) for $q\neq n+1.$ When $q=n+1$ $P_\lambda$ is already supported at the origin and has rank $1.$ This completes the proof.
\end{proof}

As an immediate corollary of the purity of the nearby and vanishing cycles, we know the ``center of minimal exponent'' of the pair $(\bb C^{2n+1}, X_n),$ an invariant introduced in \cite[Section 7.4]{SY23}. We remark that in the same paper it is shown that the minimal exponent for the hypersurface $S_n$ in $\bb P^{2n}$ is $\alpha = \frac{3}{2}.$
\begin{corollary}
	The center of minimal exponent for the pair $(\bb C^{2n+1}, X_n)$ is $X_{n-1}.$
\end{corollary}

\subsection{Explicit eigenvectors}\label{subsec:Dimca}
By \Cref{prop:Mono} we know the eigenvalues and eigenspaces of the monodromy action on $H^*(F_n,\bb C)$. However, it is possible to do even better and give a way to compute a basis for each $H^{n+1-d}(F_n)$ consisting of eigenvectors of $T.$ We give an outline of the strategy here and actually carry it out in the case $n=2.$\\
\indent For the moment, let $f$ be an arbitrary homogeneous polynomial of degree $N$ on $Y = (\bb C^*)^\ell\times \bb C^n,$ and consider the complex $(\Omega_Y^{\bullet},D_f)$ whose terms are just the usual sheaves of differential forms
\begin{equation}\label{eqn:dRKcplx}
    \qquad\mathcal O_{Y}\to\Omega_{Y}^1\to\cdots\to\Omega_{Y}^n
\end{equation}
with differential given by $D_f(\omega) = d\omega + df\wedge \omega.$ We will call this complex the \textbf{de Rham-Koszul complex} for $f,$ since the differential is the sum of the usual de Rham and Koszul differentials. In \cite[Sections 6.1-6.2]{Dim92}, Dimca shows that when $Y = \bb C^n,$ the cohomology of $(\Omega_Y^\bullet,D_f)$ is the same as the (reduced) cohomology of the Milnor fiber. He also shows that the eigenvalues of the monodromy operator are easy to read off from the cohomology of this complex. Here is how it is done. We say a $k$-form is \textbf{homogeneous of degree $d+k$} if it can be written as a sum of $k$-forms of the form
\[
    h(x_1,\ldots,x_n)dx_{i_1}\wedge\cdots \wedge dx_{i_k},
\]
where $h(x_1,\ldots,x_n)$ is a homogeneous polynomial of degree $d$ and $dx_{i_1}\wedge\cdots \wedge dx_{i_k}$ is a basic $k$-form in the coordinates $x_1,\ldots x_n.$ For each $a\in\{0,\ldots, N-1\}$ we let $\Omega_{Y,a}^\bullet$ be the subcomplex of $\Omega_Y^\bullet$ spanned by the homogeneous forms of degree $k$ where $k\equiv a\mod N.$ It's easy to see that this is a well defined subcomplex since if $\omega$ is homogeneous of degree $a,$ then
\[
	D_f(\omega)=d\omega+ df\wedge\omega
\]
where $d\omega$ and $df\wedge\omega$ are homogeneous of degree $a$ and $N+a$ respectively. We also have
\[
	(\Omega_Y^{\bullet},D_f)=\bigoplus_{a=0}^{N-1}(\Omega_{Y,a}^{\bullet},D_f).
\]
Dimca proves the following theorem; see \cite[Theorem 6.2.9]{Dim92}.
\begin{theorem}\label{thm:Dimca}
	Let $F$ be the Milnor fiber of the homogeneous polynomial $f.$ Then there is a natural isomorphism $H^{k+1}(\Omega_Y^{\bullet},D_f)\cong H^k(F,\bb C).$ Furthermore, the subspaces $H^{k+1}(\Omega_{Y,a}^{\bullet},D_f)$ map isomorphically onto the $e^{2\pi i a/N}$-eigenspace for $T$ in $H^k(F,\bb C).$
\end{theorem}

\indent Before we begin with the computation at hand, we need some lemmas.
\begin{lemma}\label{lem:dRKOnC}
    Let $g(z)=z^{m+1}$ on $\mathbb C$ with $m\ge 1$ Then
    \begin{align*}
        H^k(\Omega_{\mathbb C}^\bullet,D_g) =
        \begin{cases}
            0 & k=0,\\
            \mathbb C\langle dz,\; z dz,\;\ldots,\; z^{m-1}dz\rangle & k=1,
        \end{cases}\\
        H^k(\Omega_{\mathbb C}^\bullet(\log(\ast)),D_g) =
        \begin{cases}
            0 & k=0,\\
            \mathbb C\langle \frac{1}{z}dz,\;dz,\; z dz,\ldots,\; z^{m-1}dz\rangle & k=1.
        \end{cases}
    \end{align*}
\end{lemma}
\begin{proof}
    After taking global sections, the complex $(\Omega_{\bb C}^\bullet,D_g)$ becomes the two term complex $\bb C[z]\to \bb C[z]dz.$ The differential acts by
    \begin{align*}
        D_g(1) &= mz^{m-1}dz\\
        D_g(z^k)&= (kz^{k-1}+mz^{k-m+1})dz \quad \text{for }k\ge 1.
    \end{align*}
    From this it is easy to see that $D_g$ is injective, and the cokernel is spanned by the desired elements. The computation for log forms is similar.
\end{proof}

\begin{lemma}\label{lem:ResExSeq}
    Let $Y_P$ be a stratum as in \S \ref{subsec:StratifyCn}. Let $Z=\overline{Y_P}\setminus Y_P.$ The residue exact sequences
    \begin{equation}
    \begin{tikzcd}
        0\ar[r]& \Omega^k_{\overline{Y_P}}\ar[r]& \Omega^k_{\overline{Y_P}}(\log(Z))\ar[r,"\Res"]& \Omega^{k-1}_Z\ar[r]& 0
    \end{tikzcd}
    \end{equation}
    respect the differential $D_f,$ and hence extend to an exact sequence of complexes. Moreover, each map preserves the spaces of homogeneous forms of degree $a\mod n+1$ for each $a\in\{0,\ldots, n\}.$
\end{lemma}
\begin{proof}
    The first map clearly respects the differentials $D_f$ along with the degree of the forms mod $n+1.$ To see that the residue map does as well, we can work in coordinates. If $Z$ is defined by $z=0$ on $\overline{Y_P}$ then for $\alpha$ and $\beta$ holomorphic forms on $\overline{Y_P}$ we have
    \begin{align*}
        D_f\left(\Res\left(\beta+\alpha\wedge \frac{dz}{z}\right)\right) = d\alpha + df\wedge\alpha = \Res\left(D_f\left(\beta+\alpha\wedge\frac{dz}{z}\right)\right).
    \end{align*}
    The fact that $\Res$ preserves the degrees of homogeneous forms mod $n+1$ is due to the fact that $dz/z$ is homogeneous of degree $0.$
\end{proof}

These lemmas allows us to come up with an algorithm for computing a basis for each $H^k(F_n,\bb C).$ The strategy is to simply compute the cohomology of each $\Omega^\bullet_{\bb C^{2n+1},a}$ by doing the computation for functions of the form in \Cref{lem:dRKOnC}, and then using the above lemmas as well as \Cref{cor:SimpleF} and the structure of the stratification in \Cref{subsec:StratifyCn} to assemble the cohomology in the correct way. This is essentially a more detailed version of the computation of the eigenvalues above where we work with explicit cohomology groups and exact sequences as opposed to the Hodge polynomial and its additivity property.
\begin{example}
	As an example we carry out this strategy in the case $n=2.$ In this case we have
	\begin{align*}
		f &= \det
		\begin{pNiceMatrix}
			x_0 & x_1 & x_2\\
			x_1 & x_2 & x_3\\
			x_2 & x_3 & x_4
		\end{pNiceMatrix}\\
		&= -x_3^2+2x_1x_2x_3-x_0x_3^2-x_1^2x_4+x_0x_2x_4
	\end{align*}
	is the determinant of the $3\times 3$ Hankel matrix. Let $\lambda=e^{2\pi i/3}.$ It follows from \Cref{thm:Dimca} that under the isomorphism
	\[
		H^*(F_n,\bb C)\cong H^k(\Omega_{\bb C^5}^\bullet,D_f)
	\]
	the $\lambda$-eigenspace of the left hand side corresponds to $H^k(\Omega_{\bb C^5,2}^\bullet,D_f).$ Now stratify $\bb C^5$ as in \Cref{subsec:StratifyCn} and fix a stratum $Y_P\cong (\bb C^*)^\ell\times \bb C^n,$ where $\ell =|P|.$ By \Cref{cor:SimpleF} we can change coordinates so that $f=z^d$ where $d=\gcd(P).$ Since $f$ only involves the coordinate on one factor, it is easy to see that
	\[
		(\Omega_{\overline{Y_P}^\bullet},D_f)\cong (\Omega^\bullet_{\bb C},D_f)\otimes (\Omega^\bullet_{(\bb C^*)^{\ell-1}\times \bb C^n},d).
	\]
	Furthermore, the cohomology of the right tensor factor is spanned by homogeneous forms of degree $0.$ It follows from \Cref{lem:dRKOnC} that the only $P$ for which forms of degree $1$ appear are the $P$ on which $f$ can be written as $z^d$ on $Y_P$ with $d>1.$ Since $d=\gcd(P),$ the only partition satisfying this is $P=(3),$ and by construction, the coordinate $z=-x_2;$ see \Cref{ex:N2Strata}). By \Cref{lem:dRKOnC}, we have
	\begin{align*}
		H^k(\Omega_{\overline{Y_{(3)}},1}^\bullet,D_f) &\cong \bigoplus_{a+b=1\mod 3} H^1(\Omega^\bullet_{\bb C,a},D_f)\otimes H^{k-1}(\Omega^\bullet_{\bb C^2,b},d)\\
		&=
		\begin{cases}
			\bb C\langle dx_2\rangle & k=1\\
			0 & k\neq 1.
		\end{cases}
	\end{align*}
	Let $Z=\{x_0=0\}.$ Using the residue exact sequences and \Cref{lem:ResExSeq}, we we find that the connecting homomorphisms induce an isomorphism
	\[
		\delta:H^1(\Omega_{\overline{Y_{(3)}},1}^\bullet,D_f) \cong H^3(\Omega_{Z_0,1}^\bullet,D_f)\cong H^5(\Omega_{\bb C^5,1}^\bullet,D_f).
	\]
	The connecting homomorphisms are easily made explicit. We lift $dx_2$ to the log form $\frac{1}{x_1}dx_2\wedge dx_1,$ then we apply the differential, which gives 
	\begin{align*}
		d(f|_{Z_0})\wedge \frac{1}{x_1}dx_2\wedge dx_1 = -2x_2dx_1\wedge dx_2\wedge dx_3 + x_1dx_1\wedge dx_2\wedge dx_4,
	\end{align*}
	which is a representative of a class in $H^3(\Omega_{Z_0,1}^\bullet,D_f).$ Then we do this again to pass from $Z$ to $\bb C^5,$ and we obtain the form
	\[
		\alpha_1 = \delta(dx_2) = (2x_1x_3-2x_2^2)dx_0\wedge \cdots\wedge dx_4
	\]
	which is homogeneous of degree $7.$ The computation for $\lambda^2$ is the same except we start with $x_2dx_2,$ so we get
	\[
		\alpha_2 = (2x_1x_2x_3-2x_2^3)dx_0\wedge \cdots\wedge dx_4.
	\]
	These forms $\alpha_1$ and $\alpha_2$ generate the $\lambda$- and $\lambda^2$-eigenspaces for $T$ in $H^3(F_2,\bb C).$\\
	\indent For larger Hankel determinants connecting homomorphisms are just as easy to compute if one knows $df$, so this procedure will give the correct representatives in general. However the exact sequences become much more numerous, and keeping track of the coordinate changes in a consistent way is the current obstruction to carrying this out in general.
\end{example}

\printbibliography

@Article{Ber92,
  author   = {Bertram, Aaron},
  journal  = {Journal of Differential Geometry},
  title    = {Moduli of rank-2 vector bundles, theta divisors, and the geometry of curves in projective space},
  year     = {1992},
  issn     = {0022-040X},
  number   = {2},
  pages    = {429--469},
  volume   = {35},
  mrnumber = {1158344},
  zbl      = {0787.14014},
  zbmath   = {120208},
}

@Article{DBoi81,
  author  = {Du Bois, Philippe},
  journal = {Bull. Soc. Math. France},
  title   = {Complexe de de Rham filtré d'une variété singulière},
  year    = {1981},
  issn    = {0037-9484},
  number  = {1},
  pages   = {41-81},
  volume  = {109},
}

@Article{dCM02,
  author   = {de Cataldo, Mark Andrea A. and Migliorini, Luca},
  journal  = {Annales Scientifiques de l'{\'E}cole Normale Sup{\'e}rieure. Quatri{\`e}me S{\'e}rie},
  title    = {The Hard Lefschetz Theorem and the topology of semismall maps},
  year     = {2002},
  issn     = {0012-9593},
  number   = {5},
  pages    = {759-772},
  volume   = {35},
  zbl      = {1021.14004},
  zbmath   = {1910888},
}

@Article{Gug17,
  author        = {Gugnin, Dmitry V.},
  title         = {On Integral Cohomology Ring of Symmetric Products},
  year          = {2017},
  month         = feb,
  archiveprefix = {arXiv},
  copyright     = {arXiv.org perpetual, non-exclusive license},
  eprint        = {1502.01862},
  primaryclass  = {math.AT},
  pubstate  = {Preprint},
}

@Article{Mac62,
  author   = {I. G. MacDonald},
  journal  = {Topology},
  title    = {Symmetric products of an algebraic curve},
  year     = {1962},
  issn     = {0040-9383},
  pages    = {319-343},
  volume   = {1},
  zbl      = {0121.38003},
  zbmath   = {3196715},
}

@Misc{Sri02,
  author       = {Srinivas, Vasudevan},
  howpublished = {Tata Institute of Fundamental Research},
  month        = dec,
  note         = {Lectures at MSRI},
  title        = {The Hodge characteristic},
  year         = {2002},
}

@Book{Dim92,
  author    = {Dimca, Alexandru},
  publisher = {Springer},
  title     = {Singularities and Topology of Hypersurfaces},
  year      = {1992},
  address   = {New York, NY},
  edition   = {1},
  isbn      = {978-0-387-97709-6},
  month     = apr,
  series    = {Universitext},
}

@Book{PS08,
  author    = {Peters, Chris and Steenbrink, Joseph},
  publisher = {Springer Berlin, Heidelberg},
  title     = {Mixed Hodge Structures},
  year      = {2008},
  edition   = {1},
  isbn      = {978-3-540-77015-2},
  series    = {Ergebnisse der Mathematik und ihrer Grenzgebiete},
  volume    = {52},
}

@Article{Eis88,
  author   = {Eisenbud, David},
  journal  = {American Journal of Mathematics},
  title    = {Linear sections of determinantal varieties},
  year     = {1988},
  issn     = {0002-9327},
  number   = {3},
  pages    = {541--575},
  volume   = {110},
  zbl      = {0681.14028},
  zbmath   = {4114794},
}

@Article{dCM09,
  author    = {Mark Andrea A. de Cataldo and Luca Migliorini},
  journal   = {Bulletin of the American Mathematical Society},
  title     = {The decomposition theorem, perverse sheaves and the topology of algebraic maps},
  year      = {2009},
  issn      = {0273-0979},
  number    = {4},
  pages     = {535-535},
  volume    = {46},
  publisher = {American Mathematical Society (AMS)},
}

@Article{Ste76,
	author   = {Steenbrink, Joseph},
	journal  = {Inventiones Mathematicae},
	title    = {Limits of {Hodge} structures},
	year     = {1976},
	issn     = {0020-9910},
	pages    = {229--257},
	volume   = {31},
	zbl      = {0303.14002},
	zbmath   = {3473825},
}

@Book{KS90,
  author    = {Masaki Kashiwara and Pierre Schapira},
  title     = {Sheaves on Manifolds},
  year      = {1990},
  fseries   = {Grundlehren der mathematischen Wissenschaften},
  publisher = {Springer Berlin Heidelberg},
}

@Article{Bry86,
  author       = {Brylinski, Jean-Luc},
  journal      = {Ast{\'e}risque},
  title        = {Transformations canoniques, dualit{\'e} projective, th{\'e}orie de {Lefschetz}, transformations de {Fourier} et sommes trigonom{\'e}triques},
  year         = {1986},
  pages        = {3--134},
  number       = {140-141},
  zbl          = {0624.32009},
  zbmath       = {4013020},
}

@Book{Max19,
  author    = {Maxim, Lauren{\c{t}}iu G.},
  publisher = {Springer},
  title     = {Intersection Homology \& Perverse Sheaves with Applications to Singularities},
  year      = {2019},
  isbn      = {978-3-030-27643-0; 978-3-030-27644-7},
  series    = {Grad. Texts Math.},
  volume    = {281},
  fseries   = {Graduate Texts in Mathematics},
  issn      = {0072-5285},
  zbl       = {1476.55001},
  zbmath    = {7105770},
}

@Article{SY22,
	author        = {Schnell, Christian and Yang, Ruijie},
	title         = {A log resolution for the theta divisor of a hyperelliptic curve},
	year          = {2022},
	month         = jun,
	archiveprefix = {arXiv},
	copyright     = {arXiv.org perpetual, non-exclusive license},
	eprint        = {2206.07075},
	primaryclass  = {math.AG},
	pubstate  = {Preprint},
}

@Article{SY23,
  author        = {Schnell, Christian and Yang, Ruijie},
  pubstate  = {Preprint},
  title         = {Higher multiplier ideals},
  year          = {2023},
  month         = sep,
  arxiv         = {arXiv:2309.16763},
  eprint        = {2309.16763},
  primaryclass  = {math.AG},
}

@Article{Sai88,
  author     = {Saito, Morihiko},
  journal    = {Publ. Res. Inst. Math. Sci.},
  title      = {Modules de {H}odge polarisables},
  year       = {1988},
  issn       = {0034-5318,1663-4926},
  number     = {6},
  pages      = {849--995},
  volume     = {24},
  fjournal   = {Kyoto University. Research Institute for Mathematical Sciences. Publications},
  mrclass    = {32C35 (14C30 32C38 32C42 32G99)},
  mrnumber   = {1000123},
  mrreviewer = {J.\ H. M. Steenbrink},
}

@Article{LY24,
  author   = {L\H{o}rincz, Andr\'as and Yang, Ruijie},
  title    = {V-Filtrations of semi-invariant functions on multiplicity-free spaces},
  year     = {2024},
  pubstate = {In preparation},
}

@Article{ORS23,
  author       = {Olano, Sebastian and Raychaudhury, Debaditya and Song, Lei},
  pubstate     = {Preprint},
  title        = {Singularities of secant varieties from a {Hodge} theoretic perspective},
  year         = {2023},
  arxiv        = {arXiv:2310.09391},
  zbmath       = {902345162},
}

@Book{Che22,
  author    = {Chen, Qianyu},
  publisher = {ProQuest LLC, Ann Arbor, MI},
  title     = {Limits of {H}odge {S}tructures via {D}-{M}odules},
  year      = {2022},
  isbn      = {979-8819-38509-8},
  note      = {Thesis (Ph.D.)--State University of New York at Stony Brook},
  mrclass   = {99-05},
  mrnumber  = {4464145},
  pages     = {159},
}

@Article{MP20,
  author     = {Musta\c{t}\u{a}, Mircea and Popa, Mihnea},
  journal    = {Forum Math. Sigma},
  title      = {Hodge ideals for {$\Bbb Q$}-divisors, {$V$}-filtration, and minimal exponent},
  year       = {2020},
  issn       = {2050-5094},
  pages      = {Paper No. e19, 41},
  volume     = {8},
  fjournal   = {Forum of Mathematics. Sigma},
  mrclass    = {14D07 (14F10 14J17 32S25)},
  mrnumber   = {4089396},
  mrreviewer = {Christian\ Schnell},
}

@Article{Sch73,
  author   = {Schmid, Wilfried},
  journal  = {Inventiones Mathematicae},
  title    = {Variation of {Hodge} structure: {The} singularities of the period mapping},
  year     = {1973},
  issn     = {0020-9910},
  pages    = {211--319},
  volume   = {22},
  zbl      = {0278.14003},
  zbmath   = {3437307},
}
\end{document}